\RequirePackage{fix-cm}
\documentclass[smallextended]{svjour3}       % onecolumn (second format)

\usepackage{rotating, booktabs}
\smartqed  % flush right qed marks, e.g. at end of proof
\usepackage{graphicx}
\usepackage{etex}
\usepackage{hyperref}
\usepackage{subfig}
\usepackage{amsmath,amssymb,pst-all,color,enumerate}
\usepackage[latin1]{inputenc}
\usepackage{xspace}
\usepackage{algorithm}
\usepackage{algorithmic}
\usepackage{tikz}
 \usetikzlibrary{calc} 
\usepackage{graphicx}
\hypersetup{
    colorlinks,
    linkcolor={blue},
    citecolor={blue},
    urlcolor={blue}
}

\def\norm#1{\left\|#1\right\|}
\newcommand{\VECTOR}[3]{\left(#1,\, #2,\, #3\right)\transp}
\newcommand{\VECCTOR}[6]{\left(#1,\, #2,\, #3,\, #4,\, #5,\, #6\right)\transp}
\newcommand{\R}{\mathbb R}
\newcommand{\N}{\mathbb N}

\newcommand{\mcU}{\mathcal{U}}
\newcommand{\ybar}{y_d}

\newcommand{\efeas}{\varepsilon_{feas}}

\newcommand{\Mt}{\tilde{M}}
\newcommand{\Kt}{\tilde{K}}
\newcommand{\Nt}{\tilde{N}}
\newcommand{\Phit}{\tilde{\Phi}}
\newcommand{\Cineq}{C_{\text{ineq}}}
\newcommand{\Svec}{S_{\text{vec}}}
\newcommand{\Mred}{M_{\text{red}}}
\newcommand{\Xred}{X_{\text{red}}}
\newcommand{\Wred}{W_{\text{red}}}
\newcommand{\Kred}{K_{\text{red}}}
\newcommand{\Cred}{C_{\text{red}}}
\newcommand{\Phired}{\Phi_{\text{red}}}
\newcommand{\yred}{y_{\text{red}}}
\newcommand{\hyred}{\hat{y}_{\text{red}}}

\newcommand{\Ktred}{\tilde{K}_{\text{red}}}
\newcommand{\Phitred}{\tilde{\Phi}_{\text{red}}}
\newcommand{\hC}{\hat{C}}
\newcommand{\fred}{f_{\text{red}}}
\newcommand{\mcA}{\mathcal{A}}
\newcommand{\mcB}{\mathcal{B}}
\newcommand{\mcC}{\mathcal{C}}
\newcommand{\mcD}{\mathcal{D}}
\newcommand{\mcM}{\mathcal{M}}
\newcommand{\mcP}{\mathcal{P}}
\newcommand{\mchA}{\hat{\mathcal{A}}}
\newcommand{\mchB}{\hat{\mathcal{B}}}
\newcommand{\mchC}{\hat{\mathcal{C}}}

\newcommand{\hy}{\hat{y}}
\newcommand{\Mobs}{M_{\text{obs}}}
\newcommand{\Oobs}{\Omega_{\text{obs}}}
\newcommand{\yout}{y_{\text{out}}}
\newcommand{\yredout}{y_{\text{red,out}}}
\newcommand{\mcN}{\mathcal{N}}
\newcommand{\mcNred}{\mathcal{N}_{\text{red}}}

\newcommand{\xtIPA}{x_{\text{tIPA}}}
\newcommand{\xMORtIPA}{x_{\text{MOR-tIPA}}}
\newcommand{\dt}{\delta_t}
\newcommand{\dx}[1][x]{\ensuremath{\, \mathrm{d} #1}} 	%
\newcommand{\Jt}{J}
\newcommand{\oO}{\overline{\Omega}}
\newcommand{\pmax}{p_{\max}}
\newcommand{\transp}{^{^{\scriptstyle\intercal}}} % Transponierte
\newcommand{\BnB}{branch-and-bound\xspace}
\newcommand{\cplexmiqp}{{\tt cplexmiqp}\xspace}

\begin{document}

\title{An Improved Penalty Algorithm using Model Order Reduction for MIPDECO problems with partial observations} %Mixed Integer PDE Constrained Optimization

% Authors: full names plus addresses.
\author{Dominik Garmatter \and \\
        Margherita Porcelli \and
        Francesco Rinaldi \and
        Martin Stoll
}

% Sets running headers as well as PDF title and authors
\titlerunning{An IPA using MOR for MIPDECO problems with partial observations}

\authorrunning{D. Garmatter, M. Porcelli, F. Rinaldi and M. Stoll} 

%%%%%%%%%%%%%%%%%% Document %%%%%%%%%%%%%%%%%%%%%%%%%%%%%%

\institute{Dominik Garmatter \and Martin Stoll \at
              Department of Mathematics, Chemnitz University of Technology, Germany \\
%              Tel.: +123-45-678910\\
              \email{dominik.garmatter@math.tu-chemnitz.de, martin.stoll@math.tu-chemnitz.de}           %  \\
           \and
           Margherita Porcelli \at
              Department of Mathematics  \& AM$^2$, University of Bologna, Italy\\
              ISTI--CNR,  Italy \\
%              Tel.: +123-45-678910\\
              \email{margherita.porcelli@unibo.it}          %  \\
            \and
           Francesco Rinaldi \at
              Department of Mathematics ``Tullio Levi-Civita", University of Padova, Italy\\
%              Tel.: +123-45-678910\\
              \email{rinaldi@math.unipd.it}          %  \\
}

\date{Received: date / Accepted: date}
\maketitle

\begin{abstract}
%This article deals with optimal control problems governed by a linear time-dependent partial differential equation (PDE) constraint as well as integer constraints on the control. As a result we have to deal with combinatorial difficulties as well as large-scale linear systems resulting from the PDE discretization.
This work addresses optimal control problems governed by a linear time-dependent partial differential equation (PDE)  as well as integer constraints on the control.
Moreover, partial observations are assumed in the objective function. 
The resulting problem poses several numerical challenges due to the mixture of combinatorial aspects, induced by integer variables, and large scale linear algebra issues, arising from the PDE discretization.
Since classical solution approaches such as the \BnB framework are typically overwhelmed by such large-scale problems, this work extends an improved penalty algorithm proposed by the authors, to the time-dependent setting.  The main contribution is a novel combination of an interior point method, preconditioning, and model order reduction yielding a tailored local optimization solver at the heart of the overall solution procedure. A thorough numerical investigation is carried out both for a Poisson problem as well as a convection-diffusion problem demonstrating the versatility of the approach.
\end{abstract}

% REQUIRED

% % REQUIRED
% \begin{AMS}
% 34C20, 90C06, 90C11, 93C20, 90C51
% % 90C06: Large-scale problems in mathematical programming
% % 90C11: Mixed-integer programming
% % 90C51: Interior point methods 
% % 93C20: Control/observation systems governed by partial differential equations
% % 34C20: Transformation and reduction of ordinary differential equations and systems, normal forms
% 
% %Compared to IPA 1 I did cut out 65K05: Numerical mathematical programming methods and included 34C20: Transformation and reduction of ordinary differential equations and systems, normal forms
% \end{AMS}

\section{Introduction}
\label{sec:1_MINLPpaper}
Optimal control problems with PDE constraints and additional integer (and possible other constraints) are usually referred to as mixed integer PDE-constrained optimization (MIPDECO) problems. Such problems arise in a variety of real world applications such as gas networks \cite{hahn2017mixed,Schewe_Martin_2015}, the placement of tidal and wind turbines \cite{FUNKE2014658,Zhang2014,wesselhoeft2017mixed} or power networks \cite{Goettlich2019}. Approximating their solution poses significant difficulties as being in the intersection of two fields: integer programming and PDEs. While integer optimization problems have an inherent combinatorial complexity that needs to be accounted for, PDE-constrained optimization problems have to deal with large-scale linear systems resulting from the discretization of the PDE, see, e.g., \cite{troltzsch2010optimal,OCBookHinze}.

A classical solution approach for a MIPDECO problem is to \emph{first-discretize-then-optimize}: the PDE and the control are discretized, thus resulting in the continuous MIPDECO problem being approximated by a large-scale finite-dimensional mixed-integer nonlinear programming problem (MINLP).
This approach was outlined in the previous work by the authors \cite{garmatter2019improved}, where it was also shown numerically that standard techniques such as \BnB, see, e.g., \cite{BnB_Masterpaper} for an overview, indeed struggle to solve the resulting MINLP in reasonable time (both the large amount of integer variables as well as the large PDE discretization are challenging here).

As a remedy, \cite{garmatter2019improved} introduced a novel \emph{improved penalty algorithm} (IPA) that repeatedly solves an equivalent continuous penalty reformulation of the original problem for an increasing penalty parameter, with the penalty reformulation being obtained by relaxing the integer constraints and adding a suitable penalty term to the objective function to avoid non-integer solutions.
The IPA was based on an exact penalty (EXP) algorithm \cite{Lucidi_2011} that provides a theoretical framework for when to increase the penalization and when to search for a better minimizer. The IPA deviated from the EXP algorithm by employing a probabilistic search approach to determine a new iterate. Such a search was closely connected to basin hopping or iterated local search methods, see, e.g., \cite{grosso2007population,leary2000global}. The upside of this change was that the IPA only relied on a local optimization solver, where a suitable interior point method (IPM) utilizing a tailored preconditioner for the Newton system was used in \cite{garmatter2019improved}. As a result, the IPA was able to provide either the global or a high quality local minimum for a Poisson as well as a convection-diffusion model problem.

This article focuses on extending the IPA developed in \cite{garmatter2019improved} to MIPDECO problems with a linear, time-dependent PDE constraint as well as partial observations in the objective function. In this case, the resulting discretized MINLP will definitely be of large scale. To overcome the inherent complexity of this problem, we approximate the PDE constraint using \emph{balanced truncation} model order reduction (MOR), see, e.g., \cite{antoulas2005approximation}, and then develop a suitable IPM for this \emph{reduced penalty formulation}. The IPM is again well-equipped for the problem as it:
\begin{itemize}
    \item explicitly handles the non-convexity introduced by the penalty term;
    \item incorporates a specific preconditioner to handle the linear algebra as well as the singularity due to the partial observation.
\end{itemize}
Embedding this IPM into the IPA framework then allows for the solution of large-scale MIPDECO problems and the resulting algorithm is numerically investigated, both for a time-dependent Poisson as well as a convection-diffusion problem.

While the use of MOR is standard in general optimization contexts, see, e.g., \cite{gubisch2017proper,de2011balanced,DH14b,antil2011domain}, MOR for MIPDECO problems is far less investigated, see \cite{freya2019pod} for a first result. Furthermore, applying preconditioning to a reduced system of equations has only been considered once \cite{elman2015preconditioning}, while \cite{singh2020preconditioned} considers preconditioning during the generation of reduced models.
To the knowledge of the authors, the combination of an IPM, MOR, and preconditioning has not been considered so far in the literature to handle MIPDECO problems.
 
Finally, other methods for MIPDECO problems such as Sum-up-Rounding strategies \cite{manns2018multi,LeyfferSUR},  derivative-free approaches \cite{larson2019method}, and sophisticated rounding techniques \cite{LeyfferINV}, might become too costly when adapted to tackle the large-scale problems considered in this article.

The paper is organized as follows: the time-dependent model problem, its discretization, as well as the equivalent penalty formulation are presented in Section \ref{sec:2_MINLPpaper}. Section \ref{sec:3_MINLPpaper} contains the MOR approach, including some theoretical aspects, and the interior point method. Section \ref{sec:4_MINLPpaper} reviews the IPA framework, adapts it to the time-dependent setting, and discusses different perturbation strategies for the probabilistic search approach. Section \ref{sec:5_MINLPpaper} contains the numerical investigation of the new algorithm and final conclusions are drawn in Section \ref{sec:6_MINLPpaper}.

\section{Problem formulation}
\label{sec:2_MINLPpaper}

We begin with the description of the optimal control model problem in function spaces. Following the first-discretize-then-optimize approach, we then present the discretized model problem, its continuous relaxation, and then move towards the penalty reformulation of the problem.

\subsection{Time-dependent binary optimal control problem}
\label{sec:2_1_MINLPpaper}

We begin with the description of the PDE in order to formulate the optimal control problem. Consider a bounded domain $\Omega\subset\R^2$ with Lipschitz boundary, the time interval $[0,T]$ with final time $T>0$, source functions $\phi_1,\dots ,\phi_l\in L^2(\Omega)$, and based on these the parabolic PDE: for a given control function $u: (0,T) \to \R^l : t\mapsto (u^{(1)}(t),\dots , u^{(l)}(t))\transp$ find the state $y\in L^2(0,T,H_0^1(\Omega))$ solving
\begin{align}
\label{eq:PDE_parabolicMINLPpaper}
\begin{split}
\frac{\partial}{\partial t} y(t,x) - \Delta y(t,x) = \sum_{i=1}^l u^{(i)}(t) \phi_i(x),\quad (t,x)\in(0,T)\times\Omega,\\
y(0,x) = 0,\quad x\in \oO,
\end{split}
\end{align}
where the PDE is to be understood in the weak sense.
Existence and uniqueness of a solution $y\in L^2(0,T,H_0^1(\Omega))$ of \eqref{eq:PDE_parabolicMINLPpaper} follow from the Lions-Lax-Milgram theorem.

For now, we choose to model the sources $\phi_1,\dots ,\phi_l$ as Gaussian functions with centers $\tilde{x}_1,\dots,\tilde{x}_l$ in the interior of $\Omega$. Thus, for $x\in\R^2$,
\begin{align}
\label{eq:gaussian_sources_parabolicMINLPpaper}
\phi_i(x) := \kappa e^{-\frac{\norm{x-\tilde{x}_i}_2^2}{\omega}}\quad\mbox{(Gaussian)},~ i=1,\dots, l,
\end{align}
with height $\kappa >0$ and width $\omega>0$, and we will provide further details in Section \ref{sec:5_MINLPpaper}. 
Introducing the space of binary control functions in time 
\[
\mcU := \{u\in L^\infty((0,T)\times \R^l) \mid u:(0,T) \to \{0,1\}^l\},
\]
the optimal control problem in function spaces then reads: given a desired state $\ybar\in L^2((0,T)\times\Omega)$, find a solution pair $(y,u)\in L^2(0,T,H_0^1(\Omega))\times \mcU$ of 
\begin{align}
\begin{array}{cl}
\label{eq:MINLP_cont_parabolicMINLPpaper}
\displaystyle
\min_{\substack{y\in L^2(0,T,H_0^1(\Omega))\\ u\in\mcU}} & \frac{1}{2}\int_0^T \int_{\Oobs} (y-\ybar)^2 \dx \dx[t],
%\norm{y-\ybar}_{L^2((0,T)\times\Omega)}^2
\\
\mbox{s.t.} & (y,u) \mbox{ fulfill } \eqref{eq:PDE_parabolicMINLPpaper},\mbox{ and }\sum_{i=1}^l u^i(t) \leq S\in\N ,\, \forall t\in(0,T),
\end{array}
\end{align}
where $\Oobs\subset\Omega$ is our domain of observation and the inequality constraint in \eqref{eq:MINLP_cont_parabolicMINLPpaper} is commonly referred to as a \emph{knapsack constraint}.
This problem can be interpreted as fitting a desired heating pattern $\ybar$ over a domain of observation $\Oobs$ by activating up to $S\in\N$ many sources at each point in time where the sources are distributed over $\Omega$. 
Clearly, as soon as the control is suitably discretized such that the discretized feasible set only contains finitely many controls (and since for each control there is a uniquely determined state $y$), this discretized problem will in its essence be a combinatorial problem such that existence of at least one global minimizer will be guaranteed. 

\subsection{Discretized model problem and continuous relaxation}
\label{sec:2_2_MINLPpaper}

We begin with a semidiscretization of \eqref{eq:PDE_parabolicMINLPpaper} in space via the well-known \emph{method of lines}. Introducing a conforming mesh over $\Omega$ using $N$ vertices and letting $M\in\R^{N\times N}$ and $K\in\R^{N\times N}$ denote the mass and stiffness matrices (do note that $K$ and $M$ are positive definite and $M$ is symmetric), we end up with the system of ordinary differential equations (ODEs)
\begin{align}
\label{eq:ODE_parabolicMINLPpaper}
M\frac{\partial}{\partial t} y(t) + K y(t) = M\Phi u(t),\quad t\in(0,T),\quad y(0) = 0.
\end{align}
Here, $\Phi\in\R^{N\times l}$ contains the finite element coefficients of the source functions in its columns, i.e., each column contains the evaluation of the respective source function at the $N$ vertices of the grid. Thus, $M\Phi u(t)$ with the vector-valued control function $u(t):[0,T] \to \R^l$ realizes the semidiscrete right-hand side. Finally, $y:[0,T] \to \R^N$ now contains the FEM-coefficients of the solution. 

The ODE system \eqref{eq:ODE_parabolicMINLPpaper} can now be solved with a time integration method of choice and we choose the Crank-Nicholson scheme. Introducing an equidistant time-grid with $n_t\in \N$ points and step size $\dt := \frac{T}{n_t-1}$ and letting $y_i \approx y(i \dt)\in\R^N$ as well as $u_i \approx u(i \dt)\in\R^l$ denote the corresponding approximations, the scheme reads 
\begin{align}
\label{eq:theta_scheme_parabolicMINLPpaper}
(M + \frac{\dt}{2} K)y_{i+1} = (M-\frac{\dt}{2} K)y_i + \frac{\dt}{2} M \Phi u_i + \frac{\dt}{2} M \Phi u_{i+1},
% (M + \theta \dt K)y_{i+1} = (M-(1-\theta)\dt K)y_i + \dt (1-\theta) M \Phi u_i + \dt \theta M \Phi u_{i+1},
\end{align}
for $i=0,\dots, n_t - 1$.
%, where for example $\theta = 1$ yields the implicit Euler scheme and $\theta = 1/2$ yields .
Introducing the matrices $K_1 := M +\frac{\dt}{2} K$ and $K_2 := M - \frac{\dt}{2} K$ as well as the matrices
\begin{align*}
I_1 := \begin{bmatrix}
1 & 0 &  \cdots & 0 \\
0 & 1 & \ddots & \vdots\\
\vdots & \ddots & \ddots & 0\\
0 & \cdots & 0 & 1
\end{bmatrix}\in\R^{n_t\times n_t}
\quad\mbox{and}\quad
I_2 := \begin{bmatrix}
0 & \cdots &  \cdots & 0 \\
1 & \ddots &  & \vdots\\
0 & \ddots & \ddots & \vdots\\
0 & 0 & 1 & 0
\end{bmatrix}\in\R^{n_t\times n_t},
\end{align*}
we define
\begin{align}
\Kt &:= I_1\otimes K_1 + I_2 \otimes -K_2 \in\R^{n_t\cdot N\times n_t\cdot N}\,\mbox{and}
\label{eq:ktilde}\\
\Phit &:= \frac{\dt}{2}\left( I_1 \otimes M\Phi + I_2 \otimes M\Phi\right)\R^{n_t\cdot N\times n_t\cdot N},
\end{align}
where $\otimes$ denotes the Kronecker product of matrices. Using $\tilde{K}$ and $\tilde{\Phi}$,
equation \eqref{eq:theta_scheme_parabolicMINLPpaper} for $i=0,\dots, n_t - 1$ can be written as
\begin{align}
\Kt y = \Phit u,
\end{align}
where from now on $y := \left( y_1\transp,\dots , y_{n_t}\transp\right)\transp\in\R^{n_t\cdot N}$ and $u := \left( u_1\transp,\dots , u_{n_t}\transp\right)\transp\in\R^{n_t\cdot l}$ denote the fully space-time discretized state and control vectors.

Assuming that the observation domain $\Oobs$ is aligned with the FEM grid and that it contains $p$ vertices of the grid, $\Mobs\in\R^{p\times p}$ denotes the mass matrix of $\Oobs$ and the matrix $C\in\R^{p\times N}$ then realizes the evaluation of the state on $\Oobs$. Letting ${\bf 1}_l := (1, \dots, 1)\transp \in \R^{l}$ denote the $l$-dimensional unit column vector, we define 
\begin{align}
\label{eq:obj_matrices}
\begin{split}
\Mt &:= I_1 \otimes \left(C\transp\Mobs C\right)\in\R^{n_t\cdot N\times n_t\cdot N} \quad\text{as well as}\\
\Cineq &:= I_1 \otimes {\bf 1}_l\transp \in \R^{n_t\times n_t\cdot l},\quad\mbox{and}\quad \Svec := {\bf 1}_{n_t} S.
\end{split} 
\end{align}
With this notation at hand, we can formulate the \emph{discretized optimal control problem}
\begin{align}
\begin{array}{cl}
\label{eq:MINLP_raw_MINLPpaper}	
\displaystyle
\min_{y\in\R^{n_t\cdot N}, u\in\R^{n_t\cdot l}} & \frac{1}{2}(y-\ybar)\transp \Mt (y-\ybar),\\
\mbox{s.t.} & \Kt y = \Phit u, ~ u \in \{0,1\}^{n_t\cdot l}, \Cineq u \leq \Svec.
\end{array}
\end{align}
In \eqref{eq:MINLP_raw_MINLPpaper} and for the remainder of this article, $\ybar$ represents a finite element coefficient vector instead of an actual $L^2((0,T)\times\Omega)$-function. 
Relaxing the integer constraints in \eqref{eq:MINLP_raw_MINLPpaper} yields the \emph{continuous relaxation}
\begin{equation}
\begin{array}{cl}
\label{eq:MINLP_raw_contrelax_MINLPaper}
\displaystyle
\min_{y\in\R^{n_t\cdot N}, u\in\R^{n_t\cdot l}} & \frac{1}{2}(y-\ybar)\transp \Mt (y-\ybar),\\
\mbox{s.t.} & \Kt y = \Phit u,~ u \in [0,1]^{n_t \cdot l}, ~ \Cineq u \leq \Svec.
\end{array}
\end{equation}
We reformulate both problems \eqref{eq:MINLP_raw_MINLPpaper} and \eqref{eq:MINLP_raw_contrelax_MINLPaper} in a more compact way.
\begin{lemma}
\label{lemma:X_and_W_MINLPpaper}
Introducing for $x\in\R^{n_t(N+l)}$ the functions
$$
\tilde{J}(x) := \frac{1}{2} x\transp \begin{bmatrix} \Mt & 0 \\ 0 & 0 \end{bmatrix} x - x\transp \begin{bmatrix} \Mt\ybar \\ 0 \end{bmatrix} + \frac{1}{2}\ybar\transp \Mt\ybar
$$
and $f:\R^{n_t\cdot l}\to\R^{n_t \cdot N}:u\mapsto \Kt^{-1}\Phit u$,
problems \eqref{eq:MINLP_raw_MINLPpaper} and \eqref{eq:MINLP_raw_contrelax_MINLPaper} are equivalent to
\begin{equation}
\tag{P}
\label{eq:MINLP_MINLPpaper}
\begin{split}
&\min_{x\in W} \tilde{J}(x)\quad \mbox{with the feasible set} \\ 
&W := \bigg\{x = (f(u)\transp,u\transp)\transp\in\R^{n_t(N+l)} \mathrel{\Big|} u \in \{0,1\}^{n_t\cdot l}, \Cineq u \leq \Svec\bigg\}
\end{split}
\end{equation}
and
\begin{equation}
\label{eq:MINLP_contrelax_MINLPaper}
\tag{Pcont}
\begin{split}
&\min_{x\in X} \tilde{J}(x)\quad \mbox{with the feasible set} \\ 
&X := \bigg\{x = (f(u)\transp,u\transp)\transp\in\R^{n_t(N+l)} \mathrel{\Big|} u \in [0,1]^{n_t\cdot l}, \Cineq u \leq \Svec\bigg\},
\end{split}
\end{equation}
respectively. Furthermore, $W\subset\R^{n_t(N+l)}$ is compact and $X\subset\R^{n_t(N+l)}$ is compact and convex such that \eqref{eq:MINLP_contrelax_MINLPaper} is a convex problem. 
\end{lemma}
\begin{proof}
The proof follows the same arguments as \cite[Lemma 2.2]{garmatter2019improved} and is thus omitted here.
% The equivalence of the problems in question follows from the definition of the sets $W$ and $X$ and the map $f$. 
% $W$ is obviously compact as it is a finite collection of points in $\R^{n_t(N+l)}$. The set
% $$
% \bigg\{ u \in [0,1]^{n_t\cdot l} \mathrel{\Big|} \Cineq u \leq \Svec\bigg\} = \bigg\{u_i \in [0,1]^l \mathrel{\Big|} \sum_{j=1}^l u_i^{(j)} \leq S, i =1,\dots , n_t\bigg\}
% $$
% is compact and convex as it is the $n_t$-times cartesian product of the compact convex set
% $$
% \bigg\{u\in [0,1]^l \mathrel{\Big|} \sum_{j=1}^l u^{(j)} \leq S\bigg\}.
% $$
% As the image of a compact convex set under a linear map is again compact and convex, we can conclude that $X$ is compact convex.
% Thus, the convexity of \eqref{eq:MINLP_contrelax_MINLPaper} follows from the convexity of $X$ and the convexity of $\tilde{J}$ where the matrix 
% $
% \begin{bmatrix} \Mt & 0 \\ 0 & 0 \end{bmatrix},
% $
% with $\Mt$ being positive semidefinite, is positive semidefinite.
\end{proof}

\subsection{Penalty reformulation}
\label{sec:2_3_MINLPpaper}

Starting from the continuous relaxation \eqref{eq:MINLP_raw_contrelax_MINLPaper}, we add the well-known penalty term 
\begin{equation}
\label{eq:penaltyterm_MINLPpaper}
\frac{1}{\varepsilon} \sum_{j=1}^{n_t\cdot l} u^{(j)}(1-u^{(j)})
\end{equation}
to the objective function. Obviously, this concave penalty term penalizes a non-binary control, where $\varepsilon > 0$ determines the amount of penalization. This yields the following \emph{penalty formulation}
\begin{equation}
\begin{array}{cl}
\label{eq:MINLP_raw_penalty_MINLPpaper}
\displaystyle
\min_{y\in\R^{n_t\cdot N}, u\in\R^{n_t\cdot l}} & \frac{1}{2}(y-\ybar)\transp \Mt (y-\ybar) + \frac{1}{\varepsilon} \sum_{j=1}^{n_t\cdot l} u^{(j)}(1-u^{(j)}),\\
\mbox{s.t.} & \Kt y = \Phit u,~ u \in [0,1]^{n_t\cdot l}, \Cineq u \leq \Svec.
\end{array}
\end{equation}
Following Lemma \ref{lemma:X_and_W_MINLPpaper}, \eqref{eq:MINLP_raw_penalty_MINLPpaper} can be rewritten as
\begin{equation}
\label{eq:MINLP_penalty}
\tag{Ppen}
\begin{split}
&\min_{x\in X} \Jt(x;\varepsilon),\quad\mbox{with}\\
&\Jt(x;\varepsilon) := \frac{1}{2} x\transp \begin{bmatrix} \Mt & 0 \\ 0 & -\frac{2}{\varepsilon}I_{n_t\cdot l} \end{bmatrix} x - x\transp \begin{bmatrix} \Mt\ybar \\ -\frac{1}{\varepsilon}{\bf 1}_{n_t \cdot l} \end{bmatrix} + \frac{1}{2}\ybar\transp \Mt\ybar, 
\end{split}
\end{equation}
where $I_{n_t\cdot l}\in\R^{n_t\cdot l\times n_t\cdot l}$ is the identity-matrix.
\begin{proposition}
\label{prop:Equiv_MINLPpaper}
There exists an $\tilde{\varepsilon} > 0$ such that for all $\varepsilon \in (0,\tilde{\varepsilon}]$ problems \eqref{eq:MINLP_MINLPpaper} and \eqref{eq:MINLP_penalty} have the same global minima (if there exist multiple). In this sense both problems \eqref{eq:MINLP_MINLPpaper} and \eqref{eq:MINLP_penalty} are equivalent.
\end{proposition}
\begin{proof}
From Lemma \ref{lemma:X_and_W_MINLPpaper} we know that $W$ and $X$ are compact and since $\tilde{J}$ is a quadratic function, it clearly holds that $\tilde{J} \in C^1(\R^{N+l})$. Together with the results derived in \cite[Section 3]{Lucidi_2010} all assumptions of \cite[Theorem 2.1]{Lucidi_2010} are fulfilled such that the desired statement follows.
\end{proof}
Proposition \ref{prop:Equiv_MINLPpaper} holds for a variety of concave penalty terms, see, e.g., \cite[Equations (19)-(23)]{Lucidi_2010} or \cite[Equation (21)]{Rinaldi_2009}. Nevertheless, we chose the penalty term \eqref{eq:penaltyterm_MINLPpaper} here since it is quadratic and thus the combined objective function $\Jt$ remains quadratic.

The repeated solution of the penalty formulation \eqref{eq:MINLP_penalty} for an increasing value of the penalty parameter $\varepsilon$  will be the core of our solution procedure for the overall MIPDECO problem \eqref{eq:MINLP_MINLPpaper}. This procedure will be based on the IPA algorithm proposed in \cite{garmatter2019improved} and  its extension to the time-dependent setting  is postponed to Section \ref{sec:4_MINLPpaper}. In the following section we present the main algorithmic novelty instead, that is the suitable combination of the model order reduction and the interior point method for the efficient solution of the penalty formulation \eqref{eq:MINLP_penalty}.

\section{Model Order Reduction  and Interior Point Methods}
\label{sec:3_MINLPpaper}

Solving the overall MIPDECO problem \eqref{eq:MINLP_MINLPpaper} via a penalty approach requires numerous solves of \eqref{eq:MINLP_penalty}.
Thus, an efficient method to handle the problem \eqref{eq:MINLP_penalty} for a given $\varepsilon$ is crucial to an overall effective solution procedure.
With this purpose, we present two procedures based on
interior point methods: the first one, a generalization of the interior point method (IPM) proposed in \cite{garmatter2019improved} to the time-dependent case, will be denoted full-IPM. The second one, a novel combination of model order reduction (MOR) and an IPM, will be denoted MOR-IPM and is the main contribution of this work. Both methods are equipped with a preconditioning technique that exploits the specific problem structure in \eqref{eq:MINLP_penalty}.

\subsection{Model order reduction approach}
\label{sec:3_1_MINLPpaper}

The central idea is to derive a low-dimensional approximation of the PDE constraint 
$$\Kt y = \Phit u \quad \mbox{via}\quad \Ktred \yred = \Phitred u,$$
with suitable $\Ktred\in\R^{n_t\cdot r\times n_t\cdot r}$ and $\Phitred\in\R^{n_t\cdot r\times n_t\cdot l}$ such that the \emph{reconstruction} $\hC\yred\in\R^{n_t\cdot N}$, with $\hC\in \R^{n_t\cdot N\times n_t\cdot r}$, is a good approximation to $y$.
It is clear that only the dimension of the state is reduced to $r\ll N$, where the dimension of the control (in fact the control as a whole) remains untouched.
Based on this approximation, we can then (similarly to Lemma \ref{lemma:X_and_W_MINLPpaper}) introduce the linear mapping $\fred :\R^{n_t\cdot l}\to\R^{n_t\cdot N} : u\mapsto \hC \Ktred^{-1}\Phitred u$ and formulate the reduced version of the penalty formulation
\begin{equation}
\label{eq:MINLP_red_penalty}
\tag{Ppen\textsubscript{red}}
\begin{split}
&\min_{x\in \Xred} \Jt(x;\varepsilon)\quad \mbox{with the feasible set} \\ 
&\Xred := \bigg\{x = (\fred(u)\transp,u\transp)\transp\in\R^{n_t(N+l)} \mathrel{\Big|} u \in [0,1]^{n_t\cdot l}, \Cineq u \leq \Svec\bigg\}.
\end{split}
\end{equation}
Thus, only the linear map inside the feasible set changes and the better $\fred$ approximates $f$, the closer $X$ and $\Xred$ are. In the same fashion, the reduced mixed-integer control problem can be formulated as
\begin{equation}
\label{eq:MINLP_red_MINLPpaper}
\tag{P\textsubscript{red}}
\begin{split}
&\min_{x\in \Wred} \tilde{J}(x)\quad \mbox{with the feasible set} \\ 
&\Wred := \bigg\{x = (\fred(u)\transp,u\transp)\transp\in\R^{n_t(N+l)} \mathrel{\Big|} u \in \{0,1\}^{n_t\cdot l}, \Cineq u \leq \Svec\bigg\}.
\end{split}
\end{equation}
Again, the more accurate our approximation of the PDE is, the better $\fred$ approximates $f$ and the closer $\Wred$ is to $W$. Furthermore, the reduced penalty formulation \eqref{eq:MINLP_red_penalty} links to the reduced optimal control problem \eqref{eq:MINLP_red_MINLPpaper} in the same way as \eqref{eq:MINLP_penalty} links to \eqref{eq:MINLP_MINLPpaper} in Proposition \ref{prop:Equiv_MINLPpaper}, i.e., there exists an $\tilde{\varepsilon} > 0$ such that for all $\varepsilon \in (0,\tilde{\varepsilon}]$ problems \eqref{eq:MINLP_red_MINLPpaper} and \eqref{eq:MINLP_red_penalty} have the same minimum points.

Before we elaborate on the theoretical justification of this approach, we want to actually apply our model order reduction technique of choice, the \emph{balanced truncation}, see, e.g., \cite{antoulas2005approximation}, and explicitly derive the unknown quantities inside $\fred$, i.e., $\hat C$, $\Ktred$ and $\Phitred$.

\subsubsection{Balanced Truncation and reduced state system}
\label{sec:3_1_1_MINLPpaper}

Since the balanced truncation (BT) is a model order reduction technique for linear time-invariant (LTI) systems, we have to refer
to the semidiscretized ODE-system described in \eqref{eq:ODE_parabolicMINLPpaper}. We reformulate the system and add an output equation (that corresponds to the evaluation of the state on the domain of observation $\Oobs$) to fit the formulation within the standard BT literature, that is
\begin{align}
\begin{split}
\label{eq:ODE_parabolic_2_MINLPpaper}
M\frac{\partial}{\partial t} y(t) &= - K y(t) + M\Phi u(t),\quad t\in(0,T),\quad y(0) = 0,\\
\yout(t) &= Cy(t).
\end{split}
\end{align}
Note that the addition of the output equation is natural here since only the state values inside $\Oobs$ are of interest for the objective function.
After the application of the BT to this system, one can then apply a time-integration method to the resulting reduced system of ODEs to obtain $\Ktred$, $\Phitred$, and $\hC$ and thus $\fred$.

Equation \eqref{eq:ODE_parabolic_2_MINLPpaper} is an LTI system in generalized state-space form, such that we apply the generalized BT, see, e.g., \cite{saak2009efficient,badia2006balanced}. We briefly recapitulate the key steps. Our aim is to construct projection matrices $T_1\in\R^{r\times N}$ and $T_2\in\R^{N\times r}$ such that 
\begin{align}
\label{eq:red_LTI}
\begin{split}
\Mred &:= T_1 M T_2 \in \R^{r\times r},\qquad \Kred := T_1 K T_2\in \R^{r\times r},\\
\Phired &:= T_1 M \Phi\in \R^{r\times l}, \qquad \Cred := CT_2 \in \R^{p\times r},
\end{split}
\end{align}
yielding the reduced LTI system. First, we require factorizations $P = RR\transp\in \R^{N\times N}$ and $Q=LL\transp\in \R^{N\times N}$ of the solutions of the following generalized Lyapunov equations
\begin{align}
\label{eq:LyapunovEQ}
\begin{split}
&-K P M\transp - M P K\transp + M\Phi\Phi\transp M\transp = 0,\\
&-K\transp Q M - M\transp Q K + C\transp C = 0.
\end{split}
\end{align}
It is well-known that $P$ and $Q$ are positive semi-definite, such that these factorizations exist ($R$ and $L$ are often called "Cholesky" factors of $P$ and $Q$ even if they are not Cholesky factors in the strict sense). With these factors at hand, we calculate the singular value decomposition (SVD) of $L\transp M R = U\Sigma V\transp$ and mention that up to now, all these steps can be performed in a one-time offline fashion. 

Now, we choose a reduced dimension $r\ll N$ and based on this, we split the SVD with respect to this dimension $r$ as
\begin{align}
L\transp M R = U\Sigma V\transp = \begin{bmatrix}
U_1 & U_2
\end{bmatrix} \begin{bmatrix} \Sigma_1 & 0\\ 0 & \Sigma_2 
\end{bmatrix} 
\begin{bmatrix}
V_1\transp\\V_2\transp
\end{bmatrix} 
\end{align}
with $\Sigma_1 := \mathrm{diag}(\sigma_1,\dots , \sigma_r)$ and $\Sigma_2 := \mathrm{diag}(\sigma_{r+1},\dots , \sigma_N)$, where $\sigma_r > \sigma_{r+1}$ and $\sigma_j,~ j=1,\dots, N$ are the so-called \emph{Hankel singular values} of the system \eqref{eq:ODE_parabolicMINLPpaper}. Based on this truncated SVD, we define the projection matrices
$$
T_1 := \Sigma_1^{-1/2} V_1\transp R\transp\quad\mbox{and}\quad T_2 := L U_1 \Sigma_1^{-1/2}
$$
such that we obtain the \emph{reduced model} as in \eqref{eq:red_LTI}.
As a result, we obtain the reduced LTI system for the \emph{reduced state} $\hyred(t)\in\R^r$
\begin{align}
\begin{split}
\label{eq:redODE_parabolicMINLPpaper}
\Mred\frac{\partial}{\partial t} \hyred(t) &= - \Kred \hyred(t) + \Phired u(t),\quad t\in(0,T),\quad \hyred(0) = 0,\\
\yredout(t) &= \Cred \hyred(t),
\end{split}
\end{align}
which only depends on the reduced dimension $r\ll N$, and note that the control dimension remains untouched.
Similar to Section \ref{sec:2_2_MINLPpaper}, we apply the Crank-Nicholson scheme to the state equation of \eqref{eq:redODE_parabolicMINLPpaper}, which can then again be written in an all at once formulation using Kronecker-product matrices. Letting $y_{i,\mathrm{red}} \approx \hyred( i \dt)\in\R^r$, we collect these approximations in $\yred := \left(y_{1,\mathrm{red}}\transp,\dots, y_{n_t,\mathrm{red}}\transp\right)\transp\in\R^{n_t\cdot r}$ and obtain 
\begin{align}
\label{eq:red_fully_disc_PDE}
\Ktred \yred = \Phitred u,
\end{align}
where
\begin{align}\label{eq:KPhitildered}
\begin{split}
\Ktred &:= I_1\otimes K_{1,\mathrm{red}} + I_2 \otimes -K_{2,\mathrm{red}} \in\R^{n_t\cdot r\times n_t\cdot r},\\
\Phitred &:= \frac{\dt}{2}\left( I_1 \otimes \Phired + I_2 \otimes \Phired\right)\R^{n_t\cdot r\times n_t\cdot l},
\end{split}
\end{align}
with $K_{1,\mathrm{red}} := \Mred + \frac{\dt}{2}\Kred$ and $K_{2,\mathrm{red}} := \Mred - \frac{\dt}{2} \Kred$. 
Finally, we define
\begin{equation}\label{eq:Chat}
\hat{C} = I_1 \otimes T_2
\end{equation}such that we obtain the \emph{reconstruction} $\fred(u) = \hC\yred \in\R^{n_t \cdot N}$, 
% \begin{align}
% \label{eq:rec_parabolicMINLPpaper}
% \yrec := \hC\yred = \fred(u)\in\R^{n_t \cdot N},
% \end{align}
which then approximates $f(u) =  \Kt^{-1}\Phit u$. Thus, all quantities in the reduced optimal control problem \eqref{eq:MINLP_red_penalty} are now known and an IPM for the problem can be derived. We note that the matrix $\Cred$ does not appear here, since $\Cred = C T_2$ where the $T_2$
part is integrated in $\hC$ and the $C$ part is already included in $\Mt$ inside the objective function, see \eqref{eq:obj_matrices}. We also highlight that the approximation quality of the BT relies on the size of the reduced dimension $r$ and the investigation of this parameter will be the subject of the next section together with the analysis of further theoretical properties.

\subsubsection{Theoretical insights}
\label{sec:3_1_2_MINLPpaper}

We present two known theoretical results for BT and relate them to our problem. First,  a standard result targets the error between the output $\yout(t)$ of the LTI system \eqref{eq:ODE_parabolic_2_MINLPpaper}, and $\yredout,$ the output of the reduced LTI system \eqref{eq:redODE_parabolicMINLPpaper}. It requires that the system is asymptotically stable, which is the case here since both $M$ and $K$ are positive definite. For control functions $u\in L^2(0,T)$ and if the reduced LTI system \eqref{eq:redODE_parabolicMINLPpaper} was obtained via balanced truncation with reduced dimension $r \leq N$ it holds, see, e.g., \cite{antoulas2005approximation,benner2014model,antil2010domain}, that
\begin{align}
\label{eq:BT_error_1}
\norm{\yout - \yredout}_{L^2(0,T)} \leq 2 \norm{u}_{L^2(0,T)} \left(\sigma_{r+1} + \dots + \sigma_N\right),
\end{align}
where $\sigma_{r+1} + \dots + \sigma_N$ is the sum of the truncated Hankel singular values.
As a result, the approximation quality of the balanced truncation depends on the size of this sum and the LTI system \eqref{eq:ODE_parabolic_2_MINLPpaper} can be well-approximated if the Hankel singular values are quickly decaying. If the decay in the singular values is very slow, the reduced dimension has to be chosen comparably large to still ensure a good approximation. But a large $r$ negatively impacts the computational time required to solve the reduced system \eqref{eq:red_fully_disc_PDE} (since it is dense) and one might not even gain a speed-up if $r$ is too large. 
%Before further commenting on this, 
We state the second result from the literature, see \cite[Corollary 1]{antil2010domain}.

The result gives an error bound for the error between $u_*$, the solution of a generic quadratic optimal control problem
\begin{equation}
\label{eq:OCP_err_bound_1}
\begin{array}{cl}
\displaystyle\min_{u(t)} &\frac 1 2 \int_0^T \norm{\mcC y(t) + \mcD u(t) - y_d(t)}^2 \dx[t], \\
\mbox{ s.t. } 
& \mcM  \frac{\partial}{\partial t} y(t) = \mcA y(t) + \mcB u(t),\quad t\in(0,T),\quad y(0) = y_0,
\end{array}
\end{equation}
and $\hat{u}_*$, the solution of a corresponding reduced optimal control problem 
\begin{equation}
\label{eq:OCP_err_bound_2}
\begin{array}{cl}
\displaystyle\min_{u(t)} &\frac 1 2 \int_0^T \norm{\mchC \hy(t) + \mcD u(t) - y_d(t)}^2 \dx[t], \\
\mbox{ s.t. } 
& \frac{\partial}{\partial t} \hy(t) = \mchA \hy(t) + \mchB u(t),\quad t\in(0,T),\quad \hy(0) = \hy_0,
\end{array}
\end{equation}
where the reduced ODE system was obtained via balanced truncation. The result furthermore requires that $\mcM$ is symmetric, positive definite, that there exists an $\alpha>0$ such that $v\transp \mcA v \leq -\alpha v\transp\mcM v$ for all $v\in\R^N$, and that the objective function of \eqref{eq:OCP_err_bound_1} is strictly convex. Then, \cite[Corollary 1]{antil2010domain} yields the bound
\begin{align}
\label{eq:BT_error_2}
\norm{u_* - \hat{u}_*}_{L^2(0,T)} \leq \frac{2}{\kappa} \left(c\norm{\hat{u}_*}_{L^2(0,T)} + \norm{\hat{z}_*}_{L^2(0,T)}\right) \left(\sigma_{r+1} + \dots + \sigma_N\right),
\end{align}
where $\kappa$ is a constant associated with the convexity of the objective of \eqref{eq:OCP_err_bound_1}, $\hat{z}_* = \mchC \hy_* + \mcD \hat{u}_* - y_d$ with $\hy_*$ being the state corresponding to  $\hat{u}_*$, and $c$ is a constant associated to the ODE system.

This result could be applied to the continuous relaxation \eqref{eq:MINLP_contrelax_MINLPaper} with $\mcC = C$, $\mcD = 0$, $\mcM = M$, $\mcA = -K$, and $\mcB = M\Phi$.
Conversely, a bound similar to \eqref{eq:BT_error_2} cannot be expected for the penalty formulation \eqref{eq:MINLP_penalty} (on which the mixed-integer approach in this work is based) as convexity of the objective function is out of reach there.

Nonetheless, we want to stress that the driving term in the bound \eqref{eq:BT_error_2} again is the sum of the remaining Hankel singular values. Thus, if this sum is small and the reduced system  provides a good approximation, one can infer that a solution of \eqref{eq:MINLP_red_penalty} is sufficiently close to the corresponding solution of \eqref{eq:MINLP_penalty}. With \eqref{eq:BT_error_1} at hand, the feasible sets $\Xred$ and $X$ should then be close enough. The solutions of the overall MIPDECO problems \eqref{eq:MINLP_red_MINLPpaper} and \eqref{eq:MINLP_MINLPpaper} should be close, or even the same, as well.

\subsection{The interior point framework}
\label{sec:3_2_MINLPpaper}
We now briefly describe the main steps of the two interior point methods that will be employed to solve the reduced \eqref{eq:MINLP_red_penalty} and full \eqref{eq:MINLP_penalty} formulations, respectively.
The derivation of the IPMs follows \cite{Gondzio_2012} and, more specifically, \cite{garmatter2019improved}.
We first observe that problems  \eqref{eq:MINLP_red_penalty} and  \eqref{eq:MINLP_penalty} can be rewritten as 
\begin{equation}
\begin{array}{cl}
\label{eq:IPMred_parabolicMINLPpaper}
\displaystyle\min_{\substack{\yred\in \R^{n_t \cdot r},u\in \R^{n_t\cdot l},\\ z\in \R^{n_t}}} & % \bar{J}(\yred,u;\varepsilon) :=
\frac 1 2 (\hC\yred-y_d)\transp \Mt (\hC\yred-y_d)  + \frac{1}{\varepsilon} ({\bf 1}_{n_t \cdot l}\transp u-u\transp u), \\
\mbox{ s.t. } 
& \Ktred \yred = \Phitred u \qquad \mbox{ and } \qquad  \Cineq u + z - \Svec =0,\\
&   \\
& 0 \le u \le 1 \qquad \mbox{ and } \qquad z \ge 0,
\end{array}
\end{equation}
and 
\begin{equation}
\begin{array}{cl}
\label{eq:IPMfull_parabolicMINLPpaper}
\displaystyle\min_{y\in \R^{n_t \cdot N},u\in \R^{n_t\cdot l}, z\in \R^{n_t}} &  %\Jt(y,u;\varepsilon) = 
\frac 1 2 (y-y_d)\transp \Mt (y-y_d)  + \frac{1}{\varepsilon} ({\bf 1}_{n_t \cdot l}\transp u-u\transp u), \\
\mbox{ s.t. } 
& \Kt y = \Phit u \qquad \mbox{ and } \qquad  \Cineq u + z - \Svec =0,\\
&   \\
& 0 \le u \le 1 \qquad \mbox{ and } \qquad z \ge 0,
\end{array}
\end{equation}
respectively, where $0 \leq z\in\R^{n_t}$ is a vector of slack variables. We recall that $\Mt$ is defined in \eqref{eq:obj_matrices} and handles the observation on the subdomain $\Oobs$.

The main idea of an IPM is the elimination of the inequality constraints on $u$ and $z$ via the introduction of corresponding logarithmic barrier functions weighted by the barrier parameter $\mu > 0$ that controls the relation between the barrier term and the original objectives.
%$\bar{J}(\yred,u;\varepsilon)$ and $\Jt(y,u;\varepsilon)$. 
Then, first-order optimality conditions are derived by applying duality theory resulting in a nonlinear system parametrized by $\mu$. For problem \eqref{eq:IPMred_parabolicMINLPpaper}
the nonlinear system takes the form 
\begin{subequations}
\label{eq:kkt}
\begin{eqnarray}
\hat C^T \Mt \hat C \yred - \hat C^T \Mt y_d + \Ktred\transp  p  &= 0, \\ %diff wrt y
\frac{1}{\varepsilon}({\bf 1}_{n_t\cdot l} - 2u) - \Phitred\transp p + \Cineq\transp q - \lambda_{u,0} + \lambda_{u,1} &= 0,\\ % diff wrt u
q - \lambda_{z,0}  = 0, %diff wrt z 
\quad \Ktred \yred - \Phitred u = 0,\quad \Cineq u + z - \Svec &= 0, %diff wrt q
\end{eqnarray}
\end{subequations}
where the Lagrange multipliers $\lambda_{u,0}$, $\lambda_{u,1}\in\R^{n_t \cdot l}$, and $\lambda_{z,0}\in\R^{n_t}$ are defined as
\begin{equation*}
(\lambda_{u,0})_i := \frac{\mu}{u_i},~ (\lambda_{u,1})_{i} := \frac{\mu}{1 - u_i}, ~ i = 1,\dots , n_t\cdot l,~  \mbox{and}~
(\lambda_{z,0})_i := \frac{\mu}{z_i},~ i=1,\dots ,n_t.
\end{equation*}
Furthermore, the bound constraints $\lambda_{u,0} \geq 0$, $\lambda_{u,1} \geq 0$, and $\lambda_{z,0} \geq 0$ then enforce the constraints on $u$ and $z$.
Here  $p\in\R^{n_t \cdot r}$ is the reduced Lagrange multiplier (or adjoint variable) associated with the reduced state equation and $q \in \R^{n_t} $ is the Lagrange multiplier associated with the equations $\Cineq u + z - \Svec =0$.

The crucial step of deriving the IPM is the application of Newton's method to the above nonlinear system. Letting $\yred$, $u$, $z$, $p$, $q$, $\lambda_{u,0}$, $\lambda_{u,1}$, and $\lambda_{z,0}$ denote the most recent Newton iterates, these are then updated in each iteration by computing the corresponding Newton steps $\Delta \yred$, $\Delta u$, $\Delta z$, $\Delta p$, $\Delta q$, $\Delta \lambda_{u,0}$, $\Delta \lambda_{u,1}$, and $\Delta \lambda_{z,0}$ through the solution of the Newton system with the following coefficient matrix
\begin{align}
\label{eq:red_NewtonSystem}
\mcNred =
\begin{bmatrix}
  \hat C^T \Mt \hat C      &    0     &   0   &   \Ktred\transp   & 0  \\
 0             &  -\frac 2 \varepsilon I_{n_t \cdot l} + \Theta_u &   0 & -\Phitred\transp  & \Cineq\transp \\
0             &  0                       &   \Theta_z   & 0           &  I_{n_t}    \\
  \Ktred          &  -\Phitred  &   0   &  0          &  0    \\
  0            &  \Cineq       &   I_{n_t}  &   0     & 0
\end{bmatrix}.
\end{align}
Here, $\Theta_u := U^{-1} \Lambda_{u,0} + (I_l - U )^{-1} \Lambda_{u,1}$, $\Theta_z := Z^{-1}\Lambda_{z,0}$, and   $U,~Z$, $\Lambda_{u,0}$, $\Lambda_{u,1}$, as well as  $\Lambda_{z,0}$ are diagonal matrices with the most recent iterates of $u$, $z$, $\lambda_{u,0}$, $\lambda_{u,1}$, and $\lambda_{z,0}$ appearing on their diagonal entries. Once the the Newton system is solved, one can compute the steps for the Lagrange multipliers via
\begin{align*}
\Delta \lambda_{u,0} & = - \lambda_{u,0} - U^{-1}( \Lambda_{u,0} \Delta u  -\mu{\bf 1}_{n_t\cdot l}),\\
\Delta \lambda_{u,1} & = - \lambda_{u,1} + (I_{n_t \cdot l}-U)^{-1} (\Lambda_{u,1} \Delta u  + \mu {\bf 1}_{n_t\cdot l}),\\
\Delta \lambda_{z,0} & = - \lambda_{z,0} - Z^{-1}( \Lambda_{z,0} \Delta z - \mu {\bf 1}_{n_t}).
\end{align*}
A general IPM implementation only involves one Newton step per iteration. Thus, after choosing suitable step-lengths so that the updated iterates remain feasible, the new iterates can be calculated and the barrier parameter $\mu$ is reduced, thus concluding one iteration of the IPM.
Finally, we report the primal and dual feasibilities 
\begin{equation*}
 \xi_p:=  \begin{bmatrix} \Ktred \yred - \Phitred u \, \\
              \Cineq u+ z- \Svec   
              \end{bmatrix},
 \xi_d:= \begin{bmatrix}
\hat C^T \Mt \hat C \yred - \hat C^T \Mt y_d + \Ktred\transp p   \\
 \frac 1 \varepsilon ({\bf 1}_{n_t \cdot l} - 2 u) -\Phitred\transp p+ \Cineq\transp q - \lambda_{u,0} + \lambda_{u,1} \\                
q- \lambda_{z,0}  
\end{bmatrix}
\end{equation*}
as well as the complementarity gap
\begin{equation*}
\xi_c :=
\begin{bmatrix}
U\lambda_{u,0} - \mu{\bf 1}_{n_t\cdot l}, & (I_{n_t\cdot l} - U) \lambda_{u,1} - \mu{\bf 1}_{n_t\cdot l}, & Z\Lambda_{z,0} - \mu {\bf 1}_{n_t}
\end{bmatrix}\transp,
\end{equation*}
where measuring the change in the norms of $\xi_p$, $\xi_d$, and $\xi_c$ allows us to monitor the convergence of the entire process.
This completes the general description of the MOR-IPM. 

The derivation of the full-IPM for problem
\eqref{eq:IPMfull_parabolicMINLPpaper} is analogous
taking into account that the adjoint variable $p\in\R^{n_t \cdot N}$ now depends on the full dimension $N$ instead of the reduced dimension $r$ (thus, we chose not to introduce extra notation). The coefficient matrix of the resulting Newton system takes the form
\begin{align}
\label{eq:NewtonSystem}
\mcN =
\begin{bmatrix}
  \Mt            &    0     &   0   &   \Kt\transp   & 0  \\
 0             &  -\frac 2 \varepsilon I_{n_t \cdot l} + \Theta_u &   0 & -\Phit\transp  & \Cineq\transp \\
0             &  0                       &   \Theta_z   & 0           &  I_{n_t}    \\
  \Kt          &  -\Phit  &   0   &  0          &  0    \\
  0            &  \Cineq       &   I_{n_t}  &   0     & 0
\end{bmatrix},
\end{align}
where the diagonal matrices $\Theta_u$ and $ \Theta_z$ are defined as for the matrix $\mcNred$ in \eqref{eq:red_NewtonSystem}.

%Regarding the Lagrange multipliers and the quantities depending on them, the only subtle change is that The IPM is complemented with the full primal and dual feasibilities 
%\begin{equation*}
% \xi_{p}^{\text{full}}:=  \begin{bmatrix} \Kt y - \Phit u \, \\
%              \Cineq u+ z- \Svec   
%              \end{bmatrix}
%~ , ~ 
% \xi_{d}^{\text{full}}:= \begin{bmatrix}
%\Mt y- \Mt y_d + \Kt\transp  p   \\
% \frac 1 \varepsilon ({\bf 1}_{n_t \cdot l}\transp - 2 u) -\Phit\transp p+ \Cineq\transp q - \lambda_{u,0} + \lambda_{u,1} \\                
%q- \lambda_{z,0}  
%\end{bmatrix},
%\end{equation*}
%and we mention that the complementarity gap remains unchanged

We note that the matrices $\mcNred$  and $\mcN$  in  \eqref{eq:red_NewtonSystem} and \eqref{eq:NewtonSystem} have the same block structure. Moreover we observe that $\Mt$ is symmetric as
by \eqref{eq:obj_matrices} it inherits the  symmetry from $M_{obs}$,
and singular; $\Ktred$ and $\Kt$ are not symmetric no matter the symmetry of the original stiffness matrix, see definitions \eqref{eq:KPhitildered} and \eqref{eq:ktilde} respectively; $\Theta_u$ and $\Theta_z >0$, while being positive definite, are typically very ill-conditioned.
Moreover, due to the term $-\frac 2 \varepsilon I_{n_t\cdot l} $, the block  $-\frac 2 \varepsilon I_{n_t\cdot l} + \Theta_u$ may be indefinite, especially for small values of $\varepsilon$.
Following suggestions in \cite[Chapter 19.3]{nocedalwright1999numericalopt} to handle nonconvexities in the objective function by promoting the computation of descent directions, we heuristically keep the diagonal matrix $-\frac 2 \varepsilon I_{n_t\cdot l} + \Theta_u$ positive definite by setting any negative values to a small positive value $\gamma > 0$. This strategy was already implemented in \cite{garmatter2019improved} demonstrating very promising numerical performance.

From a computational point of view, the burden of any IPM lies in the solution of the Newton system at each iteration. Clearly, the developed MOR-IPM is expected to be more efficient than the full-IPM since the dominant state dimension is reduced as depicted in Figure \ref{Fig:MOR_Newton_schematic}.
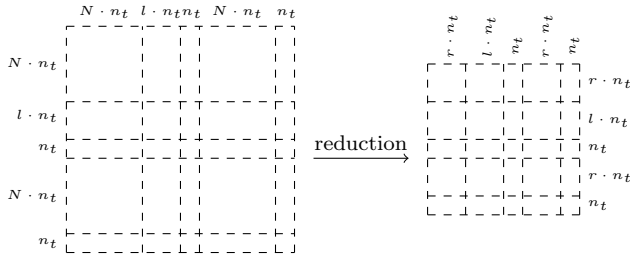
\begin{figure}[ht]
\centering
\begin{tikzpicture}[scale=0.5]
%% width of cells
\newcommand{\Hone}{0.5}
\newcommand{\Htwo}{\Hone + 2}
\newcommand{\Hthree}{\Htwo + 0.5}
\newcommand{\Hfour}{\Hthree + 1}
\newcommand{\Hfive}{\Hfour + 2}
\newcommand{\Wone}{2}
\newcommand{\Wtwo}{\Wone + 1}
\newcommand{\Wthree}{\Wtwo + 0.5}
\newcommand{\Wfour}{\Wthree + 2}
\newcommand{\Wfive}{\Wfour + 0.5}
% draw horizontal lines
\draw[dashed] (0,0) -- (\Wfive,0);
\draw[dashed] (0,\Hone) -- (\Wfive,\Hone);
\draw[dashed] (0,\Htwo) -- (\Wfive,\Htwo);
\draw[dashed] (0,\Hthree) -- (\Wfive,\Hthree);
\draw[dashed] (0,\Hfour) -- (\Wfive,\Hfour);
\draw[dashed] (0,\Hfive) -- (\Wfive,\Hfive);
% draw vertical lines
\draw[dashed] (0,0) -- (0,\Hfive);
\draw[dashed] (\Wone,0) -- (\Wone,\Hfive);
\draw[dashed] (\Wtwo,0) -- (\Wtwo,\Hfive);
\draw[dashed] (\Wthree,0) -- (\Wthree,\Hfive);
\draw[dashed] (\Wfour,0) -- (\Wfour,\Hfive);
\draw[dashed] (\Wfive,0) -- (\Wfive,\Hfive);
%% labels left
\node[left] at (0,0.5*\Hone) {\tiny $n_t$};
\node[left] at (0,3*\Hone) {\tiny $N\cdot n_t$};
\node[left] at (0,0.5*\Hthree) {\tiny $n_t$};
\node[left] at (0,0.25*\Hfour) {\tiny $l\cdot n_t$};
\node[left] at (0,3*\Hfour) {\tiny $N\cdot n_t$};
%% labels top
\node[above] at (0.5*\Wone,\Hfive) {\tiny $N\cdot n_t$};
\node[above] at (0.75*\Wtwo,\Hfive) {\tiny $l\cdot n_t$};
\node[above] at (0.9*\Wthree,\Hfive) {\tiny $n_t$};
\node[above] at (0.5*\Wfour,\Hfive) {\tiny $N\cdot n_t$};
\node[above] at (0.9*\Wfive,\Hfive) {\tiny $n_t$};
%% nice arrow
\draw[->] (\Wfive+0.5,\Htwo) -- (\Wfive+3,\Htwo) node [midway, above] {reduction};
%% reduced thing
% center coordinate
\coordinate (O) at (\Wfive+3.5,1);
% widths and heights
\newcommand{\HHone}{0.5}
\newcommand{\HHtwo}{\HHone + 1}
\newcommand{\HHthree}{\HHtwo + 0.5}
\newcommand{\HHfour}{\HHthree + 1}
\newcommand{\HHfive}{\HHfour + 1}
\newcommand{\WWone}{1}
\newcommand{\WWtwo}{\WWone + 1}
\newcommand{\WWthree}{\WWtwo + 0.5}
\newcommand{\WWfour}{\WWthree + 1}
\newcommand{\WWfive}{\WWfour + 0.5}
% draw horizontal lines
\draw[dashed] ($(O)$) -- ($(O)+(\WWfive,0)$);
\draw[dashed] ($(O)+(0,\HHone)$) -- ($(O)+(\WWfive,\HHone)$);
\draw[dashed] ($(O)+(0,\HHtwo)$) -- ($(O)+(\WWfive,\HHtwo)$);
\draw[dashed] ($(O)+(0,\HHthree)$) -- ($(O)+(\WWfive,\HHthree)$);
\draw[dashed] ($(O)+(0,\HHfour)$) -- ($(O)+(\WWfive,\HHfour)$);
\draw[dashed] ($(O)+(0,\HHfive)$) -- ($(O)+(\WWfive,\HHfive)$);
% vertical lines
\draw[dashed] ($(O)$) -- ($(O)+(0,\HHfive)$);
\draw[dashed] ($(O)+(\WWone,0)$) -- ($(O)+(\WWone,\HHfive)$);
\draw[dashed] ($(O)+(\WWtwo,0)$) -- ($(O)+(\WWtwo,\HHfive)$);
\draw[dashed] ($(O)+(\WWthree,0)$) -- ($(O)+(\WWthree,\HHfive)$);
\draw[dashed] ($(O)+(\WWfour,0)$) -- ($(O)+(\WWfour,\HHfive)$);
\draw[dashed] ($(O)+(\WWfive,0)$) -- ($(O)+(\WWfive,\HHfive)$);
%% labels right
\node[right] at ($(O)+(\WWfive,0.5*\HHone)$) {\tiny $n_t$};
\node[right] at ($(O)+(\WWfive,2*\HHone)$) {\tiny $r\cdot n_t$};
\node[right] at ($(O)+(\WWfive,0.5*\HHthree)$) {\tiny $n_t$};
\node[right] at ($(O)+(\WWfive,2*\HHthree)$) {\tiny $l\cdot n_t$};
\node[right] at ($(O)+(\WWfive,2*\HHfour)$) {\tiny $r\cdot n_t$};
%% labels top
\node[above, rotate=90] at ($(O)+(\WWone,2.5*\HHfive)$) {\tiny $r\cdot n_t$};
\node[above, rotate=90] at ($(O)+(\WWtwo,2.5*\HHfive)$) {\tiny $l\cdot n_t$};
\node[above, rotate=90] at ($(O)+(1.2*\WWthree,2*\HHfive)$) {\tiny $n_t$};
\node[above, rotate=90] at ($(O)+(\WWfour,2.5*\HHfive)$) {\tiny $r\cdot n_t$};
\node[above, rotate=90] at ($(O)+(1.2*\WWfive,2*\HHfive)$) {\tiny $n_t$};
\end{tikzpicture}
\caption{Schematic dimension reduction from the Newton equation
with $\mathcal{N}$ in \eqref{eq:NewtonSystem} to the reduced one 
with $\mathcal{N}_{red}$ in \eqref{eq:red_NewtonSystem} obtained via balanced truncation.}
\label{Fig:MOR_Newton_schematic}
\end{figure}

%From now on, we denote by MOR-IPM  and full-IPM the interior point methods
%derived above for problems \eqref{eq:IPMred_parabolicMINLPpaper} and
%\eqref{eq:IPMfull_parabolicMINLPpaper} respectively. 
\noindent We employ the following strategy to handle the linear algebra phase inside the IPMs: on the one hand we employ an inexact Krylov strategy for the solution of the Newton system and on the other hand we design a suitable preconditioner to speed up the convergence of our Krylov method of choice.
%\Comment{Should we write any justification here? I did not use the inexactness strategy in the numerics as i felt that the overall IPA accuracy was worse and the speed of the algorithm was fast already..}
%(tests showed that the generated speed-up is not worth the potential loss in accuracy).
This strategy allows to implement the IPMs in a matrix-free manner so that the matrices defined by a Kronecker product need not be explicitly formed as the corresponding products can be performed by suitable multiplication functions using the Kronecker factors. Regarding the inexactness strategy, the idea is to increase the accuracy in the solution of the Newton equation as $\mu$ decreases in order to get savings in the computational time.
%This minimizes the occurrence of so-called \emph{oversolving} in the first interior point steps. 
Global convergence results to a solution of the first-order optimality conditions for inexact IPMs can be found in \cite{bellavia1998inexact}. 
Finally, we remark that for the MOR-IPM, compared to the full-IPM, the inexactness strategy did not play an essential role and in fact in our numerical experiments the linear systems will be solved to high accuracy without affecting the overall cpu time.
%Before we go into details of our preconditioner, we mention that in the actual implementation the space-time discretized matrices $\Kt,~ \Mt,~ \Phit$ are never formed. The implementation is done in a matrix-free fashion such that the original matrices $K,~M,~\Phi$ are sufficient and suitable multiplication-functions realize necessary matrix-vector multiplications with $\Kt,~ \Mt,~ \Phit$.

\subsubsection{Preconditioning}
\label{sec:3_2_1_MINLPpaper}

Preconditioning is a crucial tool for accelerating the speed of convergence of any Krylov method. We here focus on the Newton scheme relying on the solution of the Newton system with coefficient matrix $\mcN$ given in \eqref{eq:NewtonSystem} but the same considerations can be applied to linear systems with $\mcNred$ in \eqref{eq:red_NewtonSystem} and are not reported for the sake of conciseness.
%\begin{equation}
%\mathcal{A}=
%\begin{bmatrix}
%\Mt&0&0&\Kt\transp&0\\
%0&-\frac 2 \varepsilon I_{n_t \cdot l} + \Theta_u &0&-\Phit\transp&\Cineq\transp \\
%0&0&\Theta_z&0&I_{n_t}\\
%\Kt&-\Phit&0&0&0\\
%0&\Cineq&I_{n_t}&0&0
%\end{bmatrix}.
%\end{equation}

Given the partial observation problem the matrix $\Mt$ is indeed singular. Following the strategy in \cite{pearson2020interior}, a permutation of the Newton system results in the permuted saddle point system $\left[A \ B_2^T; B_1 \ D\right]$ with blocks
$$
\tiny 
A=
\begin{bmatrix}
\Kt&-\Phit&0\\
0&-\frac 2 \varepsilon I_{n_t \cdot l} + \Theta_u &0\\
0&0&\Theta_z\\
\end{bmatrix},
B_1=
\begin{bmatrix}
\Mt&0&0\\
0&\Cineq&I_{n_t}
\end{bmatrix},
B_2=
% \begin{bmatrix}
% 0&0\\
% -\Phit\transp&\Cineq\transp \\
% 0&I_{n_t}\\
% \end{bmatrix}, \
\begin{bmatrix}
0&-\Phit&0\\
0&\Cineq&I_{n_t}\\
\end{bmatrix}, 
D=
\begin{bmatrix}
\Kt\transp&0\\
0&0
\end{bmatrix}.
$$
Our preconditioning strategy is then based on approximating the block $A$ as well as the Schur-complement, which we approximate by

% \begin{align}
% S&=
% \begin{bmatrix}
% \Kt\transp&0\\
% 0&0
% \end{bmatrix}
% -
% \begin{bmatrix}
% \Mt&0&0\\
% 0&\Cineq&I_{n_t}
% \end{bmatrix}
% \begin{bmatrix}
% \Kt&-\Phit&0\\
% 0&-\frac 2 \varepsilon I_{n_t \cdot l} + \Theta_u &0\\
% 0&0&\Theta_z\\
% \end{bmatrix}^{-1}
% \begin{bmatrix}
% 0&0\\
% -\Phit\transp&\Cineq\transp \\
% 0&I_{n_t}\\
% \end{bmatrix}.\\
% \end{align}
% and we use the approximation
% $$
% \begin{bmatrix}
% \Kt&-\Phit&0\\
% 0&-\frac 2 \varepsilon I_{n_t \cdot l} + \Theta_u &0\\
% 0&0&\Theta_z\\
% \end{bmatrix}^{-1}
% \approx
% \begin{bmatrix}
% \Kt&0&0\\
% 0&-\frac 2 \varepsilon I_{n_t \cdot l} + \Theta_u &0\\
% 0&0&\Theta_z^{-1}\\
% \end{bmatrix}^{-1}=:\hat{A},
% $$

% $$
% \begin{bmatrix}
% \Kt&-\Phit&0\\
% 0&C &0\\
% 0&0&\Theta_z\\
% \end{bmatrix}^{-1}
% \approx
% \begin{bmatrix}
% \Kt^{-1}&\Kt^{-1}\Phit C^{-1} &0\\
% 0&C^{-1} &0\\
% 0&0&\Theta_z^{-1}\\
% \end{bmatrix},
% $$
\begin{align*}
S&\approx
\begin{bmatrix}
\Kt\transp&0\\
0&0
\end{bmatrix}
-
\begin{bmatrix}
\Mt&0&0\\
0&\Cineq&I_{n_t}
\end{bmatrix}
\begin{bmatrix}
\Kt^{-1}&0&0\\
0&(-\frac 2 \varepsilon I_{n_t \cdot l} + \Theta_u)^{-1} &0\\
0&0&\Theta_z^{-1}\\
\end{bmatrix}
\begin{bmatrix}
0&0\\
-\Phit\transp&\Cineq\transp \\
0&I_{n_t}\\
\end{bmatrix}\\
&\approx
\begin{bmatrix}
\Kt\transp&0\\
0 &-\Theta_z^{-1}-\Cineq(-\frac 2 \varepsilon I_{n_t \cdot l} + \Theta_u)^{-1}\Cineq^{T},
\end{bmatrix},
\end{align*}
where in the first step we approximate the $(1,1)$-block by its diagonal and then in the last step ignore the $(2,1)$-block of the approximation to obtain a block-diagonal approximation of the Schur-complement. We embed this into the overall preconditioner obtained as 
\begin{align}
\mcP&=
% \begin{bmatrix}
% \hat{A}&0\\
% B_1&-\hat{S}
% \end{bmatrix}\\
% &=
% \begin{bmatrix}
% \Kt&0&0&0&0\\
% 0&-\frac 2 \varepsilon I_{n_t \cdot l} + \Theta_u &0&0&0\\
% 0&0&\Theta_z&0&0\\
% \Mt&0&0&-\Kt\transp&0\\
% 0&\Cineq&0&0 &-I_{n_t}+\Cineq(-\frac 2 \varepsilon I_{n_t \cdot l} + \Theta_u)^{-1}\Cineq^{T}
% \end{bmatrix}\\
% &=
\begin{bmatrix}
\Kt&0&0&0&0\\
0&(-\frac 2 \varepsilon I_{n_t \cdot l} + \Theta_u)&0&0&0\\
0&0&\Theta_z&0&0\\
\Mt&0&0&-\Kt\transp&0\\
0&\Cineq&0&0 &\Theta_z^{-1}+\Cineq(-\frac 2 \varepsilon I_{n_t \cdot l} + \Theta_u)^{-1}\Cineq^{T}
\end{bmatrix}.
\end{align}
% and
% \begin{align}
% P^{-1}
% &=
% \begin{bmatrix}
% \Kt^{-1}&0&0&0&0\\
% 0&C^{-1}&0&0&0\\
% 0&0&\Theta_z^{-1}&0&0\\
% \Kt^{-T}\Mt\Kt^{-1}&0&0&-\Kt^{-T}&0\\
% 0&(I_{n_t}-\Cineq C^{-1}\Cineq^{T})^{-1}\Cineq C^{-1}&0&0 &-(I_{n_t}-\Cineq C^{-1}\Cineq^{T})^{-1}
% \end{bmatrix}.
% \end{align}
Note that this is still a preconditioner for the permuted problem and following \cite{pearson2020interior} we can obtain its unpermuted version $\tilde{\mcP}$, given in an efficiently implemented version via
% Since we do not want to work with the permuted matrix we now consider the unpermuted version to get
% \begin{align}
% \tilde{P}^{-1}:=P^{-1}\Pi&=
% \begin{bmatrix}
% \Kt^{-1}&0&0&0&0\\
% 0&C^{-1}&0&0&0\\
% 0&0&\Theta_z^{-1}&0&0\\
% \Kt^{-T}\Mt\Kt^{-1}&0&0&-\Kt^{-T}&0\\
% 0&(I_{n_t}-\Cineq C^{-1}\Cineq^{T})^{-1}\Cineq C^{-1}&0&0 &(-I_{n_t}+\Cineq C^{-1}\Cineq^{T})^{-1}
% \end{bmatrix}	
% \left[
% 		\begin{array}{ccccc}
% 			0&0&0&I&0\\
% 			0&I&0&0&0\\
% 			0&0&I&0&0\\
% 			I&0&0&0&0\\
% 			0&0&0&0&I\\
% 		\end{array}
% 	\right]\\
% &=
% \begin{bmatrix}
% 0&0&0&\Kt^{-1}&0\\
% 0&C^{-1}&0&0&0\\
% 0&0&\Theta_z^{-1}&0&0\\
% -\Kt^{-T}&0&0&\Kt^{-T}\Mt\Kt^{-1}&0\\
% 0&(I_{n_t}-\Cineq C^{-1}\Cineq^{T})^{-1}\Cineq C^{-1}&0&0&(-I_{n_t}+\Cineq C^{-1}\Cineq^{T})^{-1}
% \end{bmatrix}	.
% \end{align}
% In order to make the implementation efficient we consider 
$$
\tilde{\mcP}^{-1}
\begin{bmatrix}
w_1\\
w_2\\
w_3\\
w_4\\
w_5\\
\end{bmatrix}=
\begin{bmatrix}
\Kt^{-1}w_4\\
C^{-1}w_2\\
\Theta_z^{-1}w_3\\
\Kt^{^{-T}}\left(-w_1+\Mt\Kt^{-1}w_4\right)\\
(\Theta_z^{-1}+\Cineq(-\frac 2 \varepsilon I_{n_t \cdot l} + \Theta_u)^{-1}\Cineq^{T})^{-1}(\Cineq C^{-1}w_2+w_5)\\
\end{bmatrix},
$$
which clearly shows that we need to solve once with $\tilde{K}$ and once with $\tilde{K}^T$.
We combine this preconditioner with the GMRES method of \cite{gmres}. We here focus on the highly relevant case of a partial observation domain, which as we already pointed out renders the matrix $\Mt$ singular. In case of a full observation, we proposed a preconditioner in \cite{garmatter2019improved} that we believe can be easily extended to the time-dependent case, see \cite{pearson2012regularization} for a preconditioning method for full observation optimization.

\section{Time-dependent Improved Penalty Algorithm (tIPA)}
\label{sec:4_MINLPpaper}

With the IPMs from the previous section at hand, we want to solve the overall MIPDECO problem. In order to do so, we adapt the improved penalty algorithm (IPA), developed in \cite{garmatter2019improved}, to this time-dependent setting. Before that, we make the following clarifying remark.

\begin{remark}
\label{rem:IPA}
\begin{enumerate}[(i)]
\item Applying the (soon described) IPA strategy to the MIPDECO problem \eqref{eq:MINLP_MINLPpaper} involves repeated solutions of the penalty formulation \eqref{eq:MINLP_penalty} where the full-IPM from Section \ref{sec:3_2_MINLPpaper} can be used.
\item In the same way, the IPA strategy can be applied to the reduced MIPDECO problem \eqref{eq:MINLP_red_MINLPpaper} which then involves solutions of the reduced penalty formulation \eqref{eq:MINLP_red_penalty} where the MOR-IPM from Section \ref{sec:3_2_MINLPpaper} can be used.
\item To avoid confusion, we will describe the IPA based on \eqref{eq:MINLP_MINLPpaper} and  \eqref{eq:MINLP_penalty}, but want to stress that the MOR version of the algorithm can be easily obtained by simply replacing the feasible sets $X$ and $W$ by $\Xred$ and $\Wred$ as well as replacing the linear map $f$ by $\fred$ throughout this section.
\item As a result, we will obtain two algorithms: one solving \eqref{eq:MINLP_MINLPpaper} and one solving \eqref{eq:MINLP_red_MINLPpaper}, where \eqref{eq:MINLP_red_MINLPpaper} approximates \eqref{eq:MINLP_MINLPpaper} and the approximation quality is based on the quality of our model order reduction.
\end{enumerate}
\end{remark}

% \subsection{Theoretical background}
% \label{sec:4_1_MINLPpaper}

\noindent We first extend the rounding strategy developed in \cite[Definition 3.4]{garmatter2019improved} to the time-dependent setting to again suitably handle the knapsack constraint in $X$ and $W$. The idea is to apply the previously developed strategy in each time step to the time-dependent control $u$.

\begin{definition}
\label{def:SR_MINLPpaper}
Letting $x = [y\transp,u\transp]\transp\in X$ and $S\in\N$, with $S\leq l$, we split up $u\in\R^{n_t\cdot l}$ into $u = (u_1\transp,\dots, u_{n_t}\transp)\transp$ with $u_i\in\R^l$ being the control coefficients representing the $i$-the time-step. We then apply the \emph{smart rounding} introduced in \cite[Definition 3.4]{garmatter2019improved} to every $u_i$, that is
\begin{itemize}
\item for $i=1,\dots ,n_t$:
\begin{itemize}
\item Let $u_{S,i}\in\R^S$ denote the $S$ largest components of $u_i$. 
\item Define $\left[u_i \right]_{SR}$ by rounding $u_{S,i}$ component-wise to the closest integer and set the remaining components to $0$.
\end{itemize}
\item Define $\left[u \right]_{SR} := \left(\left[u_1 \right]_{SR}\transp, \dots , \left[u_{n_t} \right]_{SR}\transp\right)\transp$.
\item Define $[x]_{SR} := \left(f([u]_{SR})\transp, [u]_{SR}\transp\right)\transp \in W$.
\end{itemize}
\end{definition}
We illustrate this rounding concept by the following simple example involving only control values. We will see that the smart rounding does, by definition, satisfy the knapsack constraint, while the usual rounding to the closest integer may fail to do so.

\begin{example}
Let $S=2$, $l=3$, $n_t = 2$, and let $[\cdot]$ denote the usual rounding to the closest integer. Then, for 
\[
v = \VECCTOR{0.8}{0.7}{0.1}{0.3}{0.6}{0.9}\quad\mbox{and}\quad w = \VECCTOR{0.63}{0.62}{0.61}{0.3}{0.6}{0.9}
\]
it is $v_1 = \VECTOR{0.8}{0.7}{0.1}$ and $v_2 = \VECTOR{0.3}{0.6}{0.9}$ such that
\[
[v]_{SR} = ([v_1]_{SR}\transp, [v_2]_{SR}\transp)\transp = \VECCTOR{1}{1}{0}{0}{1}{1} = [v],
\] 
but with $w_1 = \VECTOR{0.63}{0.62}{0.61}$ and $w_2 = \VECTOR{0.3}{0.6}{0.9}$ it is 
\[
[w]_{SR} = ([w_1]_{SR}\transp, [w_2]_{SR}\transp)\transp = \VECCTOR{1}{1}{0}{0}{1}{1} \neq [w] = \VECCTOR{1}{1}{1}{0}{1}{1}.
\]
\end{example}

\noindent In \cite{garmatter2019improved}, the starting point for the development of the IPA was an exact penalty (EXP) algorithm initially reported in \cite{Lucidi_2011}. Such an EXP algorithm can, analogously to \cite{garmatter2019improved}, be formulated for the time-dependent setting presented in this article. 
Furthermore, the convergence property for this EXP algorithm is analogous to the one derived in \cite[Prop. 3.6]{garmatter2019improved}, where the only necessary theoretical update is an equivalent of \cite[Prop. 3.5]{garmatter2019improved} for the time-dependent setting. This can easily be obtained: the first half of the the proof of \cite[Prop. 3.5]{garmatter2019improved} directly carries over and in the second half, when arguing how any $\tilde{z}\in W$ can be obtained from $\bar z\in W$, one has to consider additional cases due to the fact that, in the time-dependent setting, the knapsack constraint might be satisfied as an equality in some timesteps and  as an inequality in some other timesteps.

Since the IPA slightly deviates from the EXP algorithm (such that the convergence properties do not directly apply to, but rather support the IPA) and to keep the manuscript length healthy, we decided to spare the details regarding the EXP algorithm and its convergence property. Instead, we adapt the IPA to the time-dependent setting in the following section, review its properties and discuss potential perturbation strategies.

\subsection{The algorithm and its details}
\label{sec:4_2_MINLPpaper}

%We shortly summarize the ideas of the IPA that was  developed in \cite{garmatter2019improved}, report the time-dependent algorithm, and provide a discussion of its details. 

The key idea of the IPA was to suitably adapt the framework of the EXP algorithm  reported in \cite[Algorithm 3.1]{garmatter2019improved} and originally developed in \cite{Lucidi_2011}. The EXP algorithm repeatedly solves the penalty formulation \eqref{eq:MINLP_penalty} and provides a theoretical framework that tell us when to increases the amount of penalization in the objective function and when to search for a better minimizer. In order to obtain its theoretical convergence properties, the EXP algorithm, at each iteration, requires the use of a global optimization solver. This makes the algorithm impractical in a large-scale PDE constrained optimization setting. As a consequence, the IPA employed one main change:
the next iterate in the IPA only has to reduce the objective function (in the EXP algorithm, it had to be a global minimum up to a tolerance $\delta$). This next iterate is searched for via an a probabilistic approach combining a tailored local search strategy with a perturbation of the current iterate (see Sub-Algorithm \ref{algo:Perturbation} below).
% \item This probabilistic iterative search strategy (reported in Sub-Algorithm \ref{algo:Perturbation} below) then also contains a heuristic that signals the overall algorithm that no better iterate can be found which then terminates the overall algorithm.
%\end{itemize}
%Due to these changes, the IPA was able to find local minima of high quality or even the global minimum even for large-scale problems. 
For the sake of completeness, we report the \emph{time-dependent improved penalty algorithm} (tIPA), i.e., the combination of Algorithms \ref{algo:some_Bologna_Algorithm} and \ref{algo:Perturbation}, where the time-dependent nature lies in the smart rounding introduced in Definition \ref{def:SR_MINLPpaper} as well as the underlying model problems \eqref{eq:MINLP_MINLPpaper} and \eqref{eq:MINLP_penalty} with their feasible sets $W$ and $X$, respectively.

\begin{algorithm}[ht]
\caption{tIPA($x^0\in X$, $\varepsilon^0 >0$, $\sigma\in (0,1)$, $\pmax\in\N$)}
\begin{algorithmic}[1]
\label{algo:some_Bologna_Algorithm}
\STATE{$n = 0$, $x^n = x^0$, $\varepsilon^n = \varepsilon^0$}
\STATE{\textbf{Step 1.} Call Algorithm \ref{algo:Perturbation}($x^n,~\pmax,~\varepsilon^n$) to generate a new iterate $x^{n+1}$}.
\STATE{\textbf{Step 2.}}
\IF{$x^{n+1}\notin W$ \textbf{and} $\Jt(x^{n+1};\varepsilon^n) - \Jt([x^{n+1}]_{SR};\varepsilon^n) \leq \varepsilon^n \norm{x^{n+1} - [x^{n+1}]_{SR}}_2$}
\STATE{$\varepsilon^{n+1} = \sigma\varepsilon^n$}
\ELSE
\STATE{$\varepsilon^{n+1} = \varepsilon^n$}
\ENDIF
\STATE{\textbf{Step 3.}}
\IF{$x^n = x^{n+1}$}
\RETURN{$[x^{n+1}]_{SR}$}
\ELSE
\STATE{Set $n = n+1$ and go to Step 1.}
\ENDIF
\end{algorithmic}
\end{algorithm}

\begin{algorithm}[ht]
\renewcommand{\thealgorithm}{\ref{algo:some_Bologna_Algorithm}.a} 
\caption{Reduction via perturbation($x\in X$, $\pmax\in\N$, $\varepsilon >0$)}
\begin{algorithmic}[1]
\label{algo:Perturbation}
\STATE{$x^{init} = x$}
\FOR{$j=1,\dots, \pmax$}
\STATE{Use a local optimization solver to determine a solution $x^{loc}$ of \eqref{eq:MINLP_penalty} for $\varepsilon$ using $x^{init}$ as initial guess.}
\IF{$\Jt(x^{loc};\varepsilon) < \Jt(x;\varepsilon)$}
\RETURN{$x^{loc}$}
\ELSE
\STATE{Generate a point $x^{pert} = \text{Perturbation}(x^{loc})$  and set $x^{init} = x^{pert}$.}
\ENDIF
\ENDFOR
\RETURN{$x$}
\end{algorithmic}
\end{algorithm}

\noindent In the following, we list some of the key features that the tIPA inherits from the IPA, as they are structurally identical. For a more detailed discussion and interpretation of the algorithm we refer to \cite[Section 3.2]{garmatter2019improved}.
\begin{itemize}
\item The tIPA terminates via line $11$ as soon as the iteration limit $\pmax$ is reached inside Algorithm \ref{algo:Perturbation} at Step 1. Thus, the choice of $\pmax$ and the perturbation strategy determine the quality of the solution found by the tIPA. We will discuss our perturbation strategies in the second part of this section.
\item The tIPA is expected to have a two-phase behavior: in the first phase, the penalization is increased due to line $5$ of Algorithm \ref{algo:some_Bologna_Algorithm} until a feasible integer iterate $x^{n+1}\in W$ is found and in this phase the for-loop of Algorithm \ref{algo:Perturbation} should terminate in the first iteration. In the second phase, Algorithm \ref{algo:Perturbation} is then the driving force in finding better points that provide further reductions in the objective function. 
\item A new iterate 
%$x^{n+1} = \left[y^{n+1},u^{n+1}\right]\transp$ 
is always feasible with $x^{n+1}\in X$. Thus, $x^{n+1}\notin W$ in line $4$ of Algorithm \ref{algo:some_Bologna_Algorithm} can, in a practical  implementation, be replaced by 
$$\norm{u^{n+1} - [u^{n+1}]_{SR}}_\infty > \efeas$$
with a feasibility tolerance $\efeas$. Thus, it is reasonable to return $[x^{n+1}]_{SR}$ such that the control of our output iterate is always integer and satisfies the knapsack constraint.
\end{itemize}

\noindent Inside the tIPA, we use the full-IPM developed in Section \ref{sec:3_2_MINLPpaper} to obtain a (local) solution of \eqref{eq:MINLP_MINLPpaper} in Algorithm \ref{algo:Perturbation}. As mentioned in Remark \ref{rem:IPA}, the tIPA can also be formulated for the reduced problem \eqref{eq:MINLP_red_MINLPpaper}, where the MOR-IPM from Section \ref{sec:3_2_MINLPpaper} is then used inside Algorithm \ref{algo:Perturbation} to obtain a (local) solution of \eqref{eq:MINLP_red_MINLPpaper} and we call the resulting algorithm the MOR-tIPA.

The perturbation performed in line 7 plays an important role in the overall strategy. Some tailored perturbation, depending on the problem one intends to solve, might be more beneficial in the end. We hence want to conclude this section with a discussion on the perturbation strategies that we will employ inside Algorithm \ref{algo:Perturbation} during our numerical investigation in the next section.  We present two perturbation strategies: the first one being the extension of \cite[Algorithm 2.b]{garmatter2019improved} to the time-dependent setting, that is we \emph{flip} $\theta\in\N$ many sources in each time step of the control to generate the perturbed control. The corresponding state is then calculated afterwards and the details are described in Algorithm \ref{algo:flipping_1}.

\begin{algorithm}[ht]
\renewcommand{\thealgorithm}{\ref{algo:some_Bologna_Algorithm}.b} 
\caption{Perturbation($x\in X$)}
\begin{algorithmic}[1]
\label{algo:flipping_1}
\STATE{Split $x=(y\transp,u\transp)\transp$ into the state $y\in\R^{n_t\cdot N}$ and control $u = (u_1\transp,\dots, u_{n_t}\transp)\transp\in\R^{n_t\cdot l}$. Define $u^{pert} := u$.}
\FOR{$i=1,\dots, n_t$}
\STATE{Find $I_{S}$, the set of indices of the entries of $u_i$ that are larger than $\frac{1}{2}$.}
\FOR{$j=1,\dots, \min\{\vert I_S\vert,\theta\}$}
\STATE{Randomly select $\hat{\imath}\in I_{S}$.}
\STATE{Define $I_{adj}$ the set of indices corresponding to sources \emph{adjacent} to $\tilde{x}_{\hat{\imath}}$.}
\STATE{Randomly select $\hat{\imath}_{adj}\in I_{adj}$.}
\STATE{Set $\left(u_i^{pert}\right)_{\hat{\imath}}$ to a randomly chosen value smaller than $\frac{1}{2}$.}
\STATE{Set $\left(u_i^{pert}\right)_{\hat{\imath}_{adj}}$ to a randomly chosen value larger than $\frac{1}{2}$.}
\STATE{Remove $\hat{\imath}$ from $I_{S}$.}
\ENDFOR
\ENDFOR
\STATE{Compute $y^{pert} = f(u^{pert})$ if called inside the tIPA or $y^{pert} = \fred(u^{pert})$ if called inside the MOR-tIPA.}
\RETURN{$x^{pert} := \left[(y^{pert})\transp,(u^{pert})\transp\right]\transp$}
\end{algorithmic}
\end{algorithm}

\noindent When Algorithm \ref{algo:flipping_1} is called inside the tIPA, $x$ is equal to the current iterate $x^n$. % = \left[y^{n},u^{n}\right]\transp$.
The algorithm then performs $n_t\cdot\theta\in\N$ \emph{flips} to the current control $u^n$ ($\theta$ flips per time step of the control), where a flip is one iteration of the inner for-loop of Algorithm \ref{algo:flipping_1}.
Before we discuss the second perturbation strategy, we report here our definition of adjacency from \cite[Definition 3]{garmatter2019improved}.

\begin{definition}
\label{def:adjind_MINLPpaper}    
Given a collection of points $x_1,\dots ,x_n \in \Omega$ and a radius $r > 0$, we define for a point $x_i$ the set of \emph{adjacent indices}
$$
I_{adj} := \{j\in \{1,\dots, n\} \mid j\neq i,~ \norm{x_i - x_j}_\infty \leq r\}.
$$
\end{definition}

\noindent As a result, the set of adjacent indices $I_{adj}$ in Algorithm \ref{algo:flipping_1} is obtained via Definition \ref{def:adjind_MINLPpaper} with the centers $\tilde{x}_1,\dots, \tilde{x}_l\in \Omega$ of our source functions as points. Assuming that they are arranged in a uniform $m\times m$ grid, a possible radius might be $r = \frac{1}{m}$.

Algorithm \ref{algo:flipping_1} performs $\theta$ flips per time step, which may be disadvantageous: for large $n_t$ the total amount of flips may become very large and since flips are being made in every time step, the resulting perturbation may be too far away from the current iterate to yield a productive initial guess for the local solver in Algorithm \ref{algo:Perturbation}. As a result, the overall perturbation strategy may be unable to find new iterates that improve the objective function.
We thus propose a second strategy that simply performs a fixed amount of $\theta\in\N$ flips randomly spread out over the time steps. The details are found in Algorithm \ref{algo:flipping_2}.

\begin{algorithm}[ht]
\renewcommand{\thealgorithm}{\ref{algo:some_Bologna_Algorithm}.c} 
\caption{Perturbation($x\in X$)}
\begin{algorithmic}[1]
\label{algo:flipping_2}
\STATE{Split $x=(y\transp,u\transp)\transp$ into the state $y\in\R^{n_t\cdot N}$ and control $u\in\R^{n_t\cdot l}$. Define $u^{pert} := u$.}
\STATE{Find $I_{S}$, the set containing the indices of the entries of $u$ that are larger than $\frac{1}{2}$.}
\FOR{$j=1,\dots, \min\{\vert I_S\vert,\theta\}$}
\STATE{Randomly select $\hat{\imath}\in I_{S}$.}
\STATE{Define $I_{adj}$ the set of indices corresponding to sources \emph{adjacent} to $\tilde{x}_{\hat{\imath}}$.}
\STATE{Randomly select $\hat{\imath}_{adj}\in I_{adj}$.}
\STATE{Set $\left(u_i^{pert}\right)_{\hat{\imath}}$ to a randomly chosen value smaller than $\frac{1}{2}$.}
\STATE{Set $\left(u_i^{pert}\right)_{\hat{\imath}_{adj}}$ to a randomly chosen value larger than $\frac{1}{2}$.}
\STATE{Remove $\hat{\imath}$ from $I_{S}$.}
\ENDFOR
\STATE{Compute $y^{pert} = f(u^{pert})$ if called inside the tIPA or $y^{pert} = \fred(u^{pert})$ if called inside the MOR-tIPA.}
\RETURN{$x^{pert} := \left[(y^{pert})\transp,(u^{pert})\transp\right]\transp$}
\end{algorithmic}
\end{algorithm}

\noindent With Algorithm \ref{algo:flipping_2} one has much better control over the amount of flips resulting in the perturbation $x^{pert}$. Thus, the hope is to find a balanced $\theta$ such that the resulting perturbations lie outside the current basin of attraction of the objective functional and therefore are a qualitative initial guess for the local solver in Algorithm \ref{algo:Perturbation}, resulting in a point with a potentially better objective function value. We note that with the notion of adjacency from Definition \ref{def:adjind_MINLPpaper} the output of Algorithm \ref{algo:flipping_2} does again satisfy the knapsack constraint.

%By using either Algorithm \ref{algo:flipping_1} or \ref{algo:flipping_2} one hopes that the resulting perturbation $x^{pert}$ lies outside the current basin of attraction of the objective functional and therefore is a qualitative initial guess for the local solver in Algorithm \ref{algo:Perturbation} resulting in a point with a potentially better objective function value. The numerical investigation in the upcoming section has to reveal which strategy is more promising.

Although the perturbation strategies presented depend on the uniform grid of source centers used to determine the index set $I_{adj}$, we want to stress that the underlying concept of this \emph{flipping} does not depend on the chosen modelling as it was outlined in \cite[Section 4.1]{garmatter2019improved} alongside other details of our implementation that we do not repeat here.

\section{Numerical Experiments}
\label{sec:5_MINLPpaper}

In this numerical section, we first investigate the effectiveness of the model order reduction as an approximation technique and then we test the tIPA and the MOR-tIPA also in comparison with \cplexmiqp, the \BnB routine of CPLEX \cite{CPLEX} for quadratic mixed integer problems. Before that,
%Firstly, 
we introduce a second model problem  based on a convection-diffusion PDE for which most of the numerical tests will also be carried out.

Consider the original optimal control problem \eqref{eq:MINLP_cont_parabolicMINLPpaper}, but governed by the parabolic convection-diffusion PDE
\begin{align}
\label{eq:PDE_CD_MINLPpaper}
\begin{split}
\frac{\partial}{\partial t} y(t,x)-\Delta y(t,x) + w(x)\cdot\nabla y(t,x) = \sum_{i=1}^l u^i(t) \chi_i(x),\quad (t,x)\in(0,T)\times \Omega,\\
y(0,x) = 0,\quad x\in \oO,
\end{split}
\end{align}
with the constant-in-time wind vector $w(x) = (2x_2(1-x_1^2), -2x_1(1-x_2^2))\transp$ and piecewise constant source functions $\chi_1,\dots ,\chi_l\in L^2(\Omega)$ that have the same height $\kappa$ as the gaussian source functions defined in \eqref{eq:gaussian_sources_parabolicMINLPpaper}. 
Using Q1 finite elements, while also employing the Streamline Upwind Petrov-Galerkin (SUPG) \cite{brooks1982streamline} upwinding scheme as implemented in the {\tt IFISS} software package \cite{elman2007algorithm} to discretize \eqref{eq:PDE_CD_MINLPpaper} and building the relevant finite element matrices, the semidiscretization in space is achieved. Following the approach made in Section \ref{sec:2_2_MINLPpaper}, we then obtain the resulting discretized optimal control problem, its continuous relaxation and its penalty formulation such that experiments can be carried out for this model problem as well.

In the following, we refer to this problem as the \emph{convection-diffusion} problem, while the previously considered problem is the \emph{Poisson} problem.

\subsection{Numerical setting and parameter choices}
\label{sec:5_1_MINLPpaper}

We present the numerical setting for the experiments  including default parameter choices for the algorithms. If different choices are used, it will be mentioned.

We choose $\Omega := [0,1]^2$ as our computational domain, $\Oobs := [0.25,0.5]^2$ as the domain of observation, and $[0,1]$ as the time horizon. Regarding the source functions, we choose $l=25$ sources with centers $\tilde{x}_1,\dots , \tilde{x}_l$ being arranged in a uniform $5\times 5$ grid with step size $\frac{1}{6}$ (resulting in a radius of $\frac{1}{5}$ for Definition \ref{def:adjind_MINLPpaper}). For the piece-wise constant sources of the convection-diffusion problem the points $\tilde{x}_1,\dots , \tilde{x}_l$ are the centers of the squares $\Omega_1,\dots, \Omega_l\subset\Omega$ that form a uniform decomposition of $\Omega$. The height of the sources is $\kappa = 100$ and the width $\omega$ of the Gaussian sources is chosen such that every source takes $5\%$ of its center-value at a neighboring center. We mention that this choice of height and width is motivated by \cite[Section 4.2]{wesselhoeft2017mixed}. The PDE \eqref{eq:PDE_parabolicMINLPpaper} is discretized using uniform piece-wise linear finite elements in space with a step size of $2^{-6}$ (unless specified otherwise) resulting in $N=4225$ vertices (the same step size is used for the aforementioned discretization of \eqref{eq:PDE_CD_MINLPpaper}). For the temporal dimension, we stick to an equidistant grid with $n_t = 40$ timesteps such that the overall problem consists of $N\cdot n_t = 169000$ continuous and $l\cdot n_t = 1000$ integer variables.

Regarding the full-IPM and the MOR-IPM, the outer interior point iteration is stopped as soon as either
$\max\{\norm{\xi_p}_2, \norm{\xi_d}_2, \norm{\xi_c}_2\} \leq 10^{-6}$
or the safeguard $\mu \leq 10^{-15}$ is triggered.
Furthermore, starting from an initial $\mu=1$ we decrease $\mu$ by the factor $0.1$ in each outer interior point iteration. 
The inexactness for the full-IPM is implemented by stopping GMRES when the norm of the unpreconditioned relative residual is below $\eta = \max\{\min\{10^{-4}, \mu\}, 10^{-10}\}$, while for the MOR-IPM we always use $\eta = 10^{-10}$.
% while for the MOR-IPM $\eta = 10^{-10}$.
Finally, the diagonal block $ -\frac 2 \varepsilon I_l + \Theta_u $ in either Newton system \eqref{eq:NewtonSystem} or \eqref{eq:red_NewtonSystem} is kept positive definite by setting any negative values to $\gamma = 10^{-6}$.

%Regarding the balanced truncation introduced in Section \ref{sec:3_1_1_MINLPpaper}, the Cholesky factors $R,~L$ of the Lyapunov equations \eqref{eq:LyapunovEQ} are directly calculated via the {\tt mess\_lyap} routine of the M-M.E.S.S. toolbox \cite{MMESS} for Matlab. 

 Regarding the balanced truncation introduced in Section \ref{sec:3_1_1_MINLPpaper}, we solve  the Lyapunov equations \eqref{eq:LyapunovEQ}  via the {\tt mess\_lyap} routine of the M-M.E.S.S. toolbox \cite{MMESS} for Matlab, using the default setting.
This routine computes low-rank approximations $\hat R \hat R^T \approx P$ and  $\hat L \hat L^T \approx Q$ and we use $\hat R$ and $\hat L$ which approximate the Cholesky factors $R$ and $L$, respectively.

Default parameters for the tIPA and MOR-tIPA are $\varepsilon^0 = 10^6$, $\sigma = 0.5$, and the feasibility tolerance $\efeas = 0.1$.
Both algorithms use the respective solution of \eqref{eq:MINLP_contrelax_MINLPaper} or \eqref{eq:MINLP_red_penalty} for $\varepsilon^0$ as initial guess. Do note that this is not necessary since both problems, for large enough $\varepsilon^0$, are usually still convex such that any initial guess would be sufficient.

Regarding \cplexmiqp, we use default options except that we set a time limit of $50$ hours and a memory limit of $32000$ megabytes for the search tree.

All experiments were conducted on a PC with 32 GB RAM and a QUAD-Core-Processor INTEL-Core-I7-4770 (4x 3400MHz, 8 MB Cache) utilizing Matlab 2021a via which CPLEX 12.9.0 was accessed.

\subsection{The experiments}
\label{sec:5_2_MINLPpaper}

\paragraph{First experiment} 

In this first experiment we determine a good choice of the reduced dimension $r$ for both model problems in question.
Using the M-M.E.S.S. toolbox, we can calculate an approximation to the first (and thus largest) $\Nt < N$ Hankel singular values $\sigma_1,\dots ,\sigma_{\Nt}$. Due to the theoretical investigations in Section \ref{sec:3_1_2_MINLPpaper}, we are interested in the quantity $\Sigma(r) := \sigma_{r+1} + \dots + \sigma_{\Nt}$ that is the dominant term in the balanced truncation error bounds. 
Figure \ref{fig:First_experiment_figure_MINLPpaper} depicts $\Sigma(r)$ over a possible reduced dimension $r=1,\dots, \Nt-1$ for both the Poisson and the convection-diffusion problem.
We note that this neglects the singular values $\sigma_{\Nt+1},\dots, \sigma_N$ but as the Hankel singular values are sorted in descending order it becomes clear from Figure \ref{fig:First_experiment_figure_MINLPpaper} that it is indeed justified to neglect them.

\begin{figure}[ht]
	\centering
	\includegraphics[height=0.2\textheight]{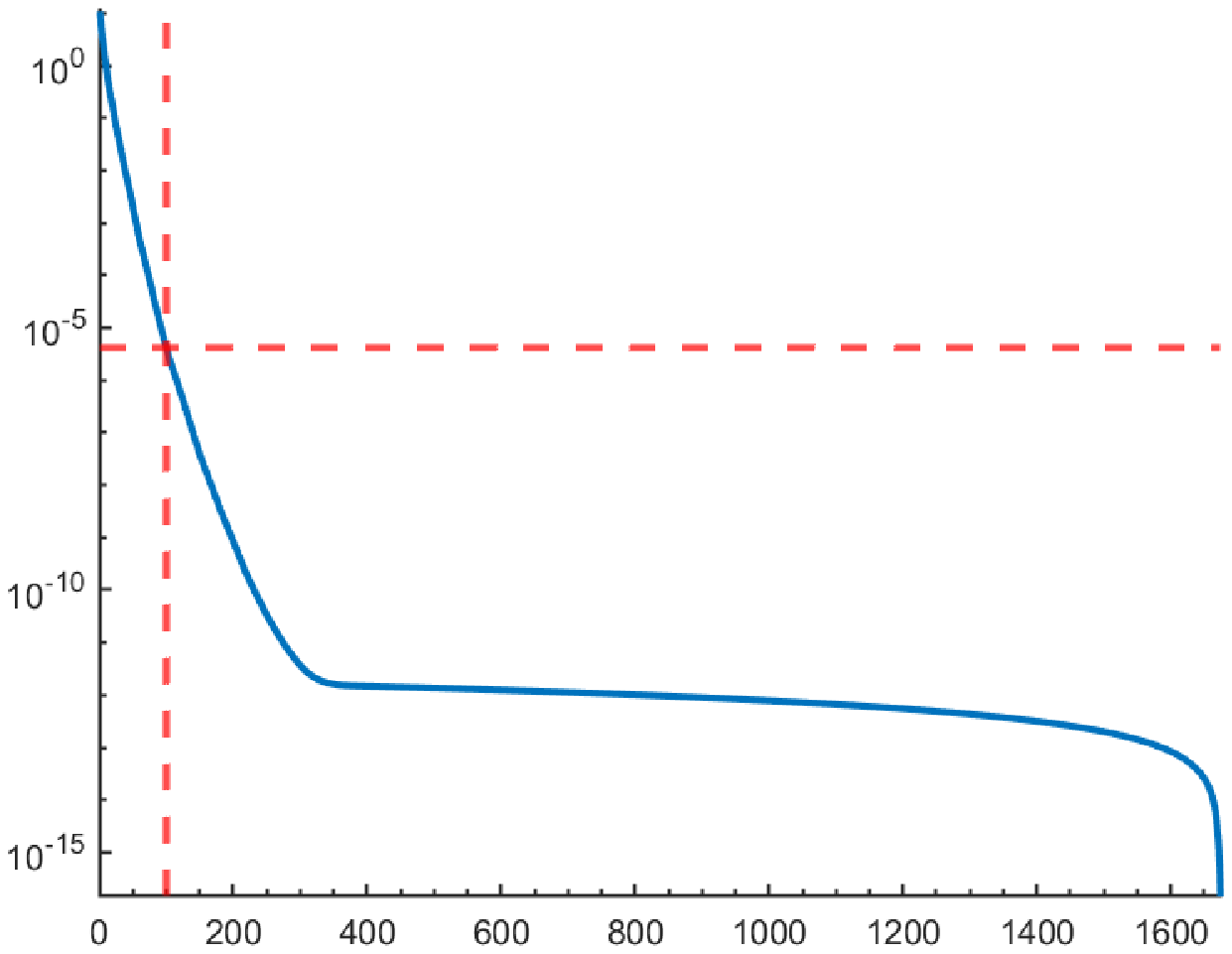}\qquad
	\includegraphics[height=0.2\textheight]{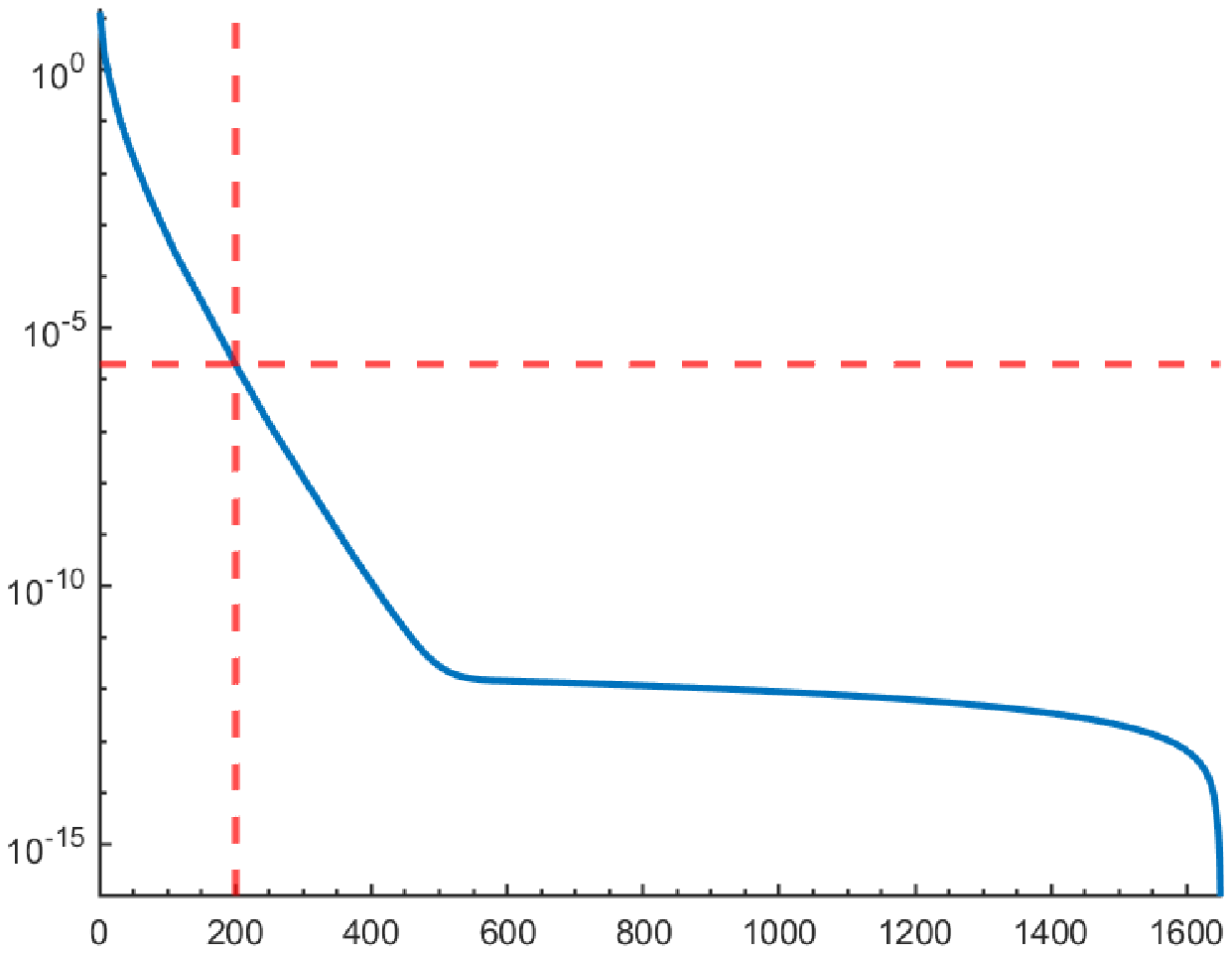}
	\caption{$\Sigma(r)$ over $r=1,\dots, \Nt-1$ for the Poisson (left) and the  convection-diffusion (right) problem. Red dotted lines indicate the chosen reduced dimensions for the later numerical experiments.}
	\label{fig:First_experiment_figure_MINLPpaper}
\end{figure}

\noindent Based on these calculations, we choose a reduced dimension $r=100$ for the Poisson problem and, to keep a similar approximation quality, $r=200$ for the convection-diffusion problem. We stress that even though the convection-diffusion problem requires a larger reduced dimension, the factor of reduction from the full state dimension $N=4225$ to the reduced dimension $200$ is still noticeable.

\paragraph{Second experiment}
Following up on the first experiment, we are now interested in the selection of the reduced dimension $r$ that is required to obtain the approximation quality of $\Sigma(r) \leq 10^{-5}$ if the FEM step size is changing. This is then an indicator on how robust our MOR approach is. We thus calculate the value of $r$ for which $\Sigma(r) \leq 10^{-5}$ for a decreasing FEM step size of $h=2^{-4},~2^{-5},~2^{-6},~2^{-7}$ for both the Poisson and the convection-diffusion problem and the result is depicted in Table \ref{table:Second_experiment_table_MINLPpaper}.

\begin{table}[ht]
\centering
%\resizebox{\textwidth}{!}{%
\begin{tabular}{c|c|c|c|c}
\toprule
$h$ & $2^{-4}$ & $2^{-5}$ & $2^{-6}$ & $2^{-7}$ \\
\hline
Poisson & 87 & 93 & 93 & 97 \\
%\hline
Conv-Diff & 102 & 146 & 172 & 190\\
\bottomrule
\end{tabular}
%}
\caption{Results of the second experiment. Reduced dimension $r$ such that $\Sigma(r) \leq 10^{-5}$ over a changing FEM step size $h$ for both model problems.}
\label{table:Second_experiment_table_MINLPpaper} 
\end{table}

\noindent Clearly, the reduced dimension $r$ required for the desired accuracy of the reduced model is robust w.r.t. the FEM step size for the Poisson problem. For the convection-diffusion problem the required reduced dimension does increase.
Internal tests showed that this is not due to the convection term (the convection would become more and more challenging for the MOR the smaller the diffusion coefficient would be) but rather due to the piece-wise constant source functions. Since we are still satisfied with the factor of reduction that is achieved, we did not further investigate this matter.

\paragraph{Third experiment} 
We carry out a first comparison of the tIPA and the MOR-tIPA, where we have two aims:
\begin{itemize}
\item Observing that both algorithms yield pretty much the same solution. Clearly, there might be slight differences due to the probabilistic nature of the IPA framework, but we want to notice that the MOR approach does not negatively influence the quality of the solution found. 
\item Investigating the behaviour of the preconditioner inside the full-IPM as well as the MOR-IPM during a tIPA iteration.
\end{itemize}

To this end, we construct a single problem instance in the following way: we generate a desired state $\ybar$ as a solution of (the discretized version of) \eqref{eq:PDE_parabolicMINLPpaper} with $S=3$ active sources in the right-hand side and the centers of these sources are randomly distributed over $[0.1,0.9]^2$, where the height and width of these sources coincides with the values presented in Section \ref{sec:5_1_MINLPpaper}. In the same fashion, a problem instance is drawn for the convection-diffusion problem. 

We now solve each problem instance with the tIPA as well as the MOR-tIPA, where we always use Algorithm \ref{algo:flipping_1} for the perturbation strategy, perturbing $\theta = 1$ source per timestep and limiting the overall perturbation cycle inside Algorithm \ref{algo:Perturbation} to $\pmax = 1000$ iterations.
We are interested in the number of nonlinear (outer) iterations (NLI) required by IPM and the average number of preconditioned GMRES iterations (aGMRES) for each value of $\varepsilon$ visited during the two versions of the IPA algorithm. The result is depicted in Figure \ref{fig:Third_experiment_figure_MINLPpaper}. 
%The remainder of Figure \ref{fig:Third_experiment_figure_MINLPpaper} contains the desired state (top right) as well as the optimal states from the tIPA (bottom left) and the MOR-tIPA (bottom right) at timestep $17$ to exemplarily display the problem.

\begin{figure}[ht]
	\centering
	\includegraphics[height=0.4\textheight]{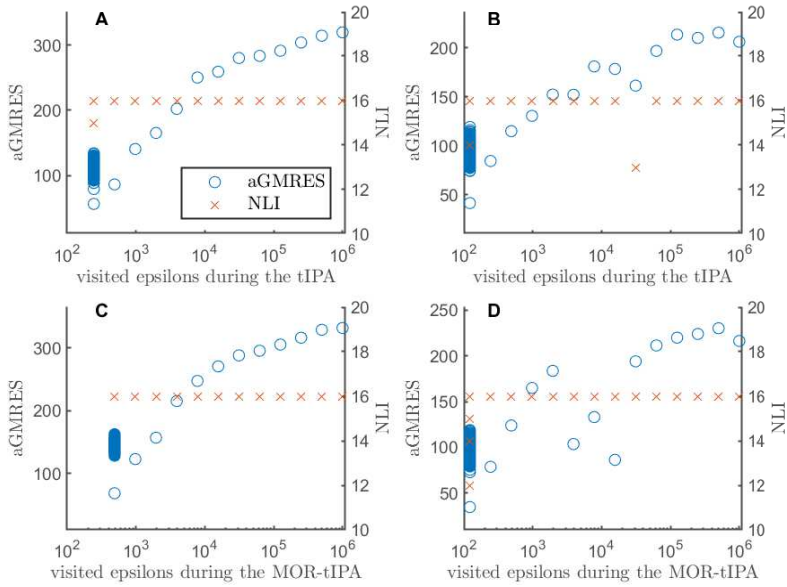}
	\caption{Number of outer IPM iterations (right y-axis) and average GMRES iterations (left y-axis) during the tIPA for the Poisson (\textbf{A}) and convection-diffusion (\textbf{B}) problem as well as the MOR-tIPA for the Poisson (\textbf{C}) and convection-diffusion (\textbf{D}) problem.}
	\label{fig:Third_experiment_figure_MINLPpaper}
\end{figure}

%\noindent Besides the visual indication of the exemplary solution quality inside $\Oobs$ (the green square in Figure \ref{fig:Third_experiment_figure_MINLPpaper} \Comment{color okay?}) and the solutions of the tIPA and the MOR-tIPA being  visually identical in the chosen time step, 
\noindent We have, for the solutions $\xtIPA$ and $\xMORtIPA$ of the Poisson problem, the objective function values
\begin{align*}
&\tilde{J}(\xtIPA) \approx 0.00496,\quad \tilde{J}(\xMORtIPA) \approx 0.00495, \quad \mbox{and}\\
&\norm{\tilde{J}(\xtIPA) - \tilde{J}(\xMORtIPA)}/\norm{\tilde{J}(\xtIPA)} \approx 0.00131.
\end{align*}
For the convection-diffusion problem, we have
\begin{align*}
&\tilde{J}(\xtIPA) \approx 0.0571,\quad \tilde{J}(\xMORtIPA) \approx 0.0556, \quad \mbox{and}\\
&\norm{\tilde{J}(\xtIPA) - \tilde{J}(\xMORtIPA)}/\norm{\tilde{J}(\xtIPA)} \approx 0.0263.
\end{align*}
Clearly, the obtained solutions for these problem instances are of high quality and the solutions obtained using the MOR-tIPA are even slightly better than the 
ones obtained with the tIPA.

Regarding the solution time, the tIPA required $55.57$ hours for the Poisson problem and $46.84$ hours for the convection-diffusion problem, where the MOR-tIPA only required $3.33$ and $11.73$ hours, respectively. The increased time in the MOR-tIPA for the convection-diffusion problem is due to the larger  dimension of the reduced problem. The results illustrate the efficiency of the MOR approach.
Possible  improvements made to the preconditioner would lead to a further reduction of the computing time, especially noticeable for the full tIPA. This is backed up by the average GMRES iterations depicted in Figure \ref{fig:Third_experiment_figure_MINLPpaper}: while they may be in a reasonable range for smaller values of $\epsilon$ (multiple blue circles for a singular value of $\varepsilon$ are due to the perturbation step), a still large amount of GMRES iterations is required in the first iterations of both tIPA and MOR-tIPA. 
% As a result, a more robust preconditionerwill make the tIPA viable and enhance the performance of the MOR-tIPA even more.

\paragraph{Fourth experiment} 

The perturbation strategy significantly impacts the quality of the overall algorithm. Therefore, we 
want to determine a qualitative strategy
in this experiment. 
To keep the manuscript length as well as the computational times healthy, this comparison is only carried out using the MOR-tIPA applied to the Poisson problem.
We distinguish the following four variants:
\begin{itemize}
\item Variant 1 (V1): the perturbation strategy from Algorithm \ref{algo:flipping_1} is used with $\theta = 1$ 
perturbation per timestep.
\item Variants 2-4 (V2-V4): the perturbation strategy from Algorithm \ref{algo:flipping_2} is used with a total of 
$\theta \in \{\lceil\frac{n_t\cdot S}{20}\rceil,~\lceil\frac{n_t\cdot S}{10}\rceil,~\lceil\frac{n_t\cdot S}{5}\rceil\}$ many perturbations. Thus, a total amount of $5\%,~10\%,~20 \%$ of the active sources is perturbed.
\end{itemize}
In each variant, we select $\pmax = 1000$ to keep a reasonable balance between computational cost of the overall algorithm and the solution quality (of course, a larger $\pmax$ will on average always improve the solution quality due to the probabilistic search approach).
For the comparison, we construct a test set of $10$ problem instances per value of $S\in\{1,2,3,4,5\}$  (we described in the previous experiment how such a problem instance is created) and it is clear that with an increased $S$ the combinatorial complexity and thus the difficulty of the MIPDECO problem increases.

We then solve this test set with the algorithms under analysis (the variants of the MOR-tIPA) and compare the results with respect to solution time and quality. For the solution time, we report 't\_av' the average solution time and for the solution quality, we choose the following two criteria.
\begin{itemize}
	\item 'min\_count': for each desired state, we check which algorithm achieved the smallest objective function value. This algorithm is then awarded a score. Surely, multiple algorithms can be awarded a score in the same run (when multiple algorithms find the same 'best' solution).
	\item 'rel\_err\_av': for each desired state, we store for each algorithm the relative error between the objective function value achieved by that algorithm and the smallest objective function value in that run (the one that was awarded a 'min\_count'-score). Only runs resulting in a non-zero relative error are taken into account when computing this average relative error.
\end{itemize}
Since the global minimum of the tackled optimization problem is not known analytically, the 'min\_count'-value tells us how often an algorithm performed best compared to the other algorithms and the average relative error is an additional measure of quality.
Furthermore, we collect 'av\_subsolvercalls' the average amount of calls of the local solver to understand how good the perturbation strategy is (the closer this value is to $\pmax = 1000$, the less effective the perturbation strategy is).
%This shall indicate how 'actively' the perturbation strategy is searching for a better minimizer. 
The results of this experiment can be found in Table \ref{table:Fourth_experiment_table_MINLPpaper}.

\begin{table}[ht]
\centering
\resizebox{\textwidth}{!}{%
\begin{tabular}{c |ccccc |ccccc |ccccc|ccccc }
\toprule
\multicolumn{1}{c}{} & \multicolumn{5}{|c}{t\_av (h)} & \multicolumn{5}{|c}{min\_count} & \multicolumn{5}{|c}{rel\_err\_av ($\%$)} & \multicolumn{5}{|c}{av\_subsolvercalls} \\
\midrule
S & 1 & 2 & 3 & 4 & 5 & 1 & 2 & 3 & 4 & 5 & 1 & 2 & 3 & 4 & 5 & 1 & 2 & 3 & 4 & 5\\
\midrule
MOR-tIPA V1 & 1.41 & 2.09 & 2.37 & 2.43 & 2.24 & 5 & 4 & 4 & 4 & 4 & 17.54 & 14.64 & 12.52 & 8.61 & 5.45 & 1014 & 1015 & 1014 & 1014 & 1014 \\
MOR-tIPA V2 & 0.72 & 1.12 & 1.26 & 1.22 & 1.35 & 7 & 10 & 10 & 8 & 9 & 5.75 & 0.00 & 0.00 & 4.58 & 1.40 & 1132 & 1427 & 1399 & 1164 & 1200\\
MOR-tIPA V3 & 0.89 & 1.49 & 1.74 & 1.82 & 1.67 & 8 & 4 & 4 & 6 & 4 & 21.20 & 11.75 & 10.05 & 6.86 & 5.31 & 1221 & 1427 & 1180 & 1124 & 1016 \\
MOR-tIPA V4 & 0.94 & 1.36 & 2.00 & 2.12 & 2.32 & 5 & 4 & 4 & 4 & 5 & 15.86 & 14.64 & 12.52 & 8.61 & 5.29 & 1017 & 1015 & 1014 & 1014 & 1075 \\
\bottomrule
\end{tabular}
}
\caption{Results of the fourth experiment. Comparison of different MOR-tIPA variants for different values of $S$.}
\label{table:Fourth_experiment_table_MINLPpaper} 
\end{table}

\noindent The major takeaway from Table \ref{table:Fourth_experiment_table_MINLPpaper} is that the second variant (using Algorithm \ref{algo:flipping_2} perturbing a total amount of $5\%$ active sources) is vastly superior to the other variants. Not only is it the fastest variant, but it also has the best solution quality: it has the largest or a very large min\_count score and very small average relative error in the instances were it does not produce the best minimizer. Going more into the details, it is very interesting to inspect the last part of Table \ref{table:Fourth_experiment_table_MINLPpaper}, i.e., av\_subsolvercalls. We observe that for both variants $1$ and $4$ the respective strategy is not actively finding better iterates since the number of calls to the local solver are close to the $\pmax = 1000$ iterations of Algorithm \ref{algo:Perturbation} that are required to terminate the overall MOR-tIPA. This strengthens the intuition we already mentioned in Section \ref{sec:4_2_MINLPpaper} that these strategies are flipping too many sources such that the resulting perturbations are useless initial guesses for the local solver (in the sense that they do not lead to better iterates of the overall MIPDECO problem).
With strategies V$3$ and V$2$ it can then be seen that more subsolvercalls are made on average indicating that the perturbation strategy is actively finding better iterates inside the MOR-tIPA leading to better overall solutions of the MIPDECO problem.

Finally, to put the results of this experiment into a better perspective, Figure \ref{fig:Fourth_experiment_figure_MINLPpaper}  contains, for each part of the test set, a Box-Plot related to the objective function of the final solutions attained by each algorithm (i.e., for each value of S the test set contains $10$ instances, such that for each algorithm a Box-Plot is created for the $10$ objective function values related to the solutions we  found).
A Box-Plot consists of several parts: the lower and upper end of box represent the 25th and the 75th percentile of the data vector represented in the respective Box-Plot, the red line inside the box depicts the median of the data and the black dashed lines extending the box  are the so called whiskers which represent the remaining data points that
are not considered outliers. The outliers are then depicted as red crosses.

\begin{figure}[ht]
	\centering
	\includegraphics[width=\textwidth, height = 0.4\textheight]{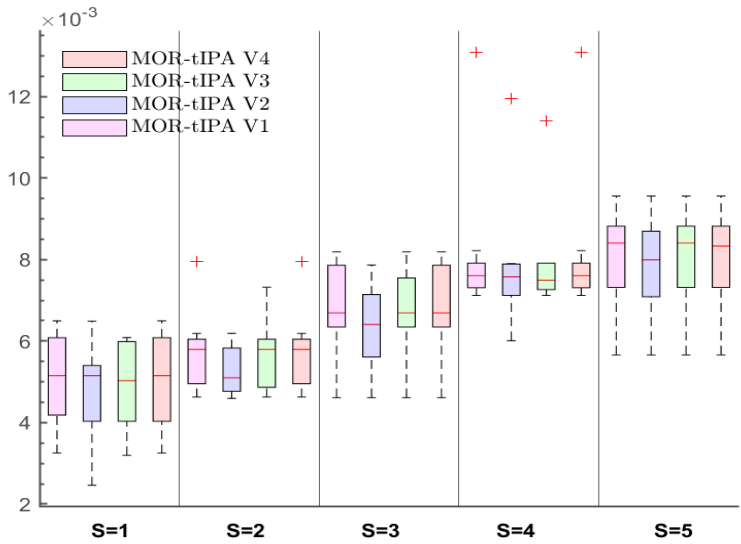}
	\caption{Results of the fourth experiment: for each part of the test set, a Box-Plot related to the objective function of the final solutions obtained by each algorithm is depicted.}
	\label{fig:Fourth_experiment_figure_MINLPpaper} 
\end{figure}

\noindent Besides showing the absolute values and thus the quality of the points obtained with the MOR-tIPA, the results of Figure \ref{fig:Fourth_experiment_figure_MINLPpaper} further strengthen the observations made in Table \ref{table:Fourth_experiment_table_MINLPpaper}: variant $2$ either achieves the smallest median (for $S=2$, $S=3$, and $S=5$) or found significantly better solutions lying in the lower whiskers than the other variants (for $S=1$ and $S=4$).

\paragraph{Fifth experiment}

In our final experiment, we want to compare the MOR-tIPA, the tIPA, and \cplexmiqp, the \BnB routine of CPLEX to verify that the MOR-tIPA is indeed the best algorithm for the MIPDECO problems tackled in this article. The experiment is carried out for the Poisson as well as the convection-diffusion problem.
We create one problem instance for each $S \in\{1,2,3,4,5\}$ (where we described in the third  experiment how such a problem instance is created) and solve it with \cplexmiqp given a time limit of $50$ hours,
as well as the tIPA and the MOR-tIPA, where both algorithms use the perturbation strategy used in variant 2 from the previous experiment.
Concerning the computational times of the tIPA, we employ a timelimit of $50$ hours to keep a fair comparison with \cplexmiqp.

Regarding the solution quality, the algorithm with the lowest objective function value is indicated with a '$\min$' in Table \ref{table:Fifth_experiment_table_MINLPpaper} 
%(or a '$\min^*$' if it was the global minimum, i.e., if \cplexmiqp found the global minimum) 
and for each other algorithm the relative error with respect to this  objective function value is then displayed. Furthermore, Table \ref{table:Fifth_experiment_table_MINLPpaper} contains the running times in hours for each algorithm in each instance, where 'TL' indicates that the time limit was reached by the given algorithm.  

\begin{table}[ht]
\centering
\resizebox{\textwidth}{!}{%
\begin{tabular}{c|c c|c c|c c|c c|c c}
\toprule

\multicolumn{11}{c}{Poisson problem}\\
\hline
\multicolumn{1}{c}{S} & \multicolumn{2}{|c}{1} & \multicolumn{2}{|c}{2} & \multicolumn{2}{|c}{3} & \multicolumn{2}{|c}{4} & \multicolumn{2}{|c}{5}\\
\hline
& time (h) & rel\_err ($\%$) & time (h) & rel\_err ($\%$) & time (h) & rel\_err ($\%$) & time (h) & rel\_err ($\%$) & time (h) & rel\_err ($\%$)\\
\hline
\cplexmiqp & \bf{TL} & 2266.13 & \bf{TL} & 4971.64 & \bf{TL} & 12495.90 & \bf{TL} & 23032.98 & \bf{TL} & 30936.74 \\
tIPA & 43.58 & $\mathrm{\bf{min}}$ & 31.18 & $\mathrm{\bf{min}}$ & 40.51 & $\mathrm{\bf{min}}$ & 25.85 & 21.20 & 26.01 & $\mathrm{\bf{min}}$ \\ 
MOR-tIPA & 1.86 & 13.45 & 1.65 & 11.05 & 1.93 & 13.25 & 1.55 & $\mathrm{\bf{min}}$ & 1.95 & 3.70 \\    
\multicolumn{11}{c}{ }\\                   
%\hline            
\multicolumn{11}{c}{Convection-diffusion problem}\\
\hline
\multicolumn{1}{c}{S} & \multicolumn{2}{|c}{1} & \multicolumn{2}{|c}{2} & \multicolumn{2}{|c}{3} & \multicolumn{2}{|c}{4} & \multicolumn{2}{|c}{5}\\
\hline
& time (h) & rel\_err ($\%$) & time (h) & rel\_err ($\%$) & time (h) & rel\_err ($\%$) & time (h) & rel\_err ($\%$) & time (h) & rel\_err ($\%$)\\
\hline
\cplexmiqp & \bf{TL} & 1934.08 & \bf{TL} & 7671.52 & \bf{TL} & 5366.21 & \bf{TL} & 3626.17 & \bf{TL} & 3743.28 \\
tIPA & 33.14 & $\mathrm{\bf{min}}$ & 37.71 & $\mathrm{\bf{min}}$ & 39.71 & 4.87 & 33.22 & 4.84 & \bf{TL} & $\mathrm{\bf{min}}$ \\ 
MOR-tIPA & 8.47 & 75.17 & 10.42 & 0.04 & 7.15 & $\mathrm{\bf{min}}$ & 6.28 & $\mathrm{\bf{min}}$ & 10.41 & $\mathrm{\bf{min}}$ \\
\bottomrule
\end{tabular}
}
\caption{Results of the fifth experiment. For each problem instance the algorithm with the lowest objective function value is indicated. The respective relative error of other algorithms as well as the solution times are furthermore reported.}
\label{table:Fifth_experiment_table_MINLPpaper} 
\end{table}

\noindent 
%\textcolor{blue}{Keeping in mind that this experiment contained only singular instances, it is safe to say that} 
Focusing on this test, we may conclude that the MOR-tIPA and the tIPA do find equally good solutions and the MOR approach does not severely deteriorate the solution quality. Moreover, the MOR-tIPA definitely outperforms the  tIPA in terms of computational time. Finally, results clearly show that \cplexmiqp is not able to find a good solution  to the large-scale problems tackled in this article in the prescribed (although large) amount of time.

\section{Conclusion \& Outlook}
\label{sec:6_MINLPpaper}
A standard MIPDECO problem with a linear time-dependent PDE constraint and a modelled control was presented and discretized. An improved penalty algorithm (IPA), developed by the authors in a previous work, was suitably adapted to the time-dependent setting, where the core of the IPA is an efficient local optimization solver paired with a probabilistic basin hopping strategy as well as an updating tool for the penalty parameter.
In order to handle the large-scale context of the time-dependent PDE constraint, we introduced a combination of an interior point method (IPM), model order reduction (MOR), and preconditioning resulting in the MOR-IPM. Integrating the MOR-IPM in the time-dependent IPA framework yielded the MOR-tIPA for the solution of the overall MIPDECO problem, which represents the main novelty of this work.

A thorough numerical investigation, dealing with a Poisson as well as a convection-diffusion problem, showed the efficiency of the model order reduction, revealed a promising perturbation strategy inside the IPA framework, and highlighted how efficiently the MOR-tIPA provides significant solutions for the difficult MIPDECO problems considered in this article (and how much \cplexmiqp, the \BnB routine of CPLEX, struggles).

On the contrary, the numerical investigation also revealed that the developed preconditioner leaves room for improvement and this will have to be considered in future work. Besides this, the next step is the development of an IPA framework (and especially an efficient local solver) for time-dependent nonlinear problems, where devising an effective model order reduction will certainly be a challenging task.

\section*{Acknowledgement}
D. Garmatter and M. Stoll acknowledge the financial support by the Federal Ministry of Education and Research of Germany (support code 05M18OCB). 
D. Garmatter thanks Dr. Jens Saak for vital discussions on the generalized balanced truncation and the M.E.S.S. toolbox. M. Porcelli is member of the INdAM Research Group GNCS
and this work was partially supported by INdAM-GNCS under Progetti di Ricerca 2020-2021.

\section*{Data availability} The data that support the findings of this study are available from the corresponding author upon request.

\overfullrule=0pt
%\printbibliography
%\bibliographystyle{siam} % We choose the &quot;plain&quot; reference style
%\bibliography{Generalbib}

\end{document}